\documentclass[12pt]{amsart}
\usepackage{new-macros}

\title
[Grasper families and barbell diffeomorphisms]
{Grasper families of spheres in $S^2\tm D^2$ and barbell diffeomorphisms of $S^1\tm S^2\tm I$}

\author[E.~Fern\'andez]{Eduardo Fern\'andez}
    \address{Department of Mathematics, University of Georgia, Athens, GA 30602, USA}
    \email{eduardofernandez@uga.edu}
\author[D.~T.~Gay]{David T. Gay}
    \address{Department of Mathematics, University of Georgia, Athens, GA 30602, USA}
    \email{dgay@uga.edu}
\author[D.~Hartman]{Daniel Hartman}
    \address{Max Planck Institute for Mathematics,Vivatsgasse 7, 53111 Bonn, Germany }
    \email{hartman@mpim-bonn.mpg.de}
\author[D. Kosanovi\'c]{Danica Kosanovi\'c}
    \address{ETH Z\"urich, Department of Mathematics, R\"amistrasse 101, 8092 Z\"urich, Switzerland}
    \email{danica.kosanovic@math.ethz.ch}

\begin{document}

\begin{abstract}
    We show that the fundamental group of framed circles in $S^1\tm D^3$ injects into the fundamental group of framed spheres in $S^2\tm D^2$, so that the cokernel is the fundamental group of framed neat disks in $D^4$. In particular, grasper families of circles give rise to countably many nontrivial families of spheres. Ambient extensions of either of these two types of families induce the same barbell diffeomorphisms of $S^1\tm S^2\tm I$.
    We give two proofs that these diffeomorphisms are nontrivial and pairwise distinct. This implies infinite generation of the abelian group of isotopy classes of diffeomorphisms of $S^1\tm S^2\tm I$ that are pseudo-isotopic to the identity, recovering a result of Singh.
\end{abstract}

\maketitle

%%%%%%%%%%%%%%%%%%%%%%
\section{Introduction}\label{sec:intro}
This paper investigates the space $\Diffp(S^1 \tm S^2 \tm I)$ of diffeomorphisms of $S^1 \tm S^2 \tm I$ that are the identity near the boundary, and its relationship to $\Emb(\nu S^1, S^1 \tm D^3)$ and $\Emb(\nu S^2, S^2 \tm D^2)$, the spaces of framed embeddings of a circle in $S^1\tm D^3$ and a sphere in $S^2\tm D^2$, respectively. The natural starting point in this study are the sequences
\begin{equation}\label{eq-intro:fibrations}
\begin{tikzcd}[row sep=5pt]
    \Diffp(S^1\tm S^2\tm I) \arrow[r,"\cup^{\nu c}_0"] 
    & \Diffp(S^1\tm D^3) \arrow[r,"\res_{\nu c}"] 
    & \Emb(\nu S^1, S^1\tm D^3),\\
    \Diffp(S^1\tm S^2\tm I) \arrow[r,"\cup^{\nu S}_1"]
    & \Diffp(S^2\tm D^2) \arrow[r,"\res_{\nu S}"]
    & \Emb(\nu S^2, S^2\tm D^2).
\end{tikzcd}
\end{equation}
The map $\res_{e}$ is a Cerf–Palais fibration: it restricts a diffeomorphism to a fixed framed circle $e=\nu c$ or framed sphere $e=\nu S$. In both cases, the fiber of $\res_{e}$ can be identified with $\Diffp(S^1\tm S^2\tm I)$ so that the map $\cup^{e}_{i}$ is the extension by the identity along the manifold $e$, attached along the boundary component $S^1\tm S^2 \tm \{i\}$ for $i=0$ or $i=1$.
\begin{mainthm}\label{mainT:diagram}
    There is a commutative diagram of \emph{trivial extensions of abelian groups}:
    \[\begin{tikzcd}[nodes={scale=0.93}]
        \ker(i^{\nu c}_0)
            \dar[tail]
            \rar[tail,two heads]{\cong}
        &  \pi_1(\Emb(\nu S^1,S^1\tm D^3);\nu c)
            \dar[tail]{\delta^{\nu c}}
        & 
        \\ 
        \pi_1(\Emb(\nu S^2,S^2\tm D^2);\nu S)
            \rar[tail]{\delta^{\nu S}}
            \dar[two heads]{i^{\nu c}_0}
        & \pi_0\Diffp(S^1\tm S^2\tm I)
            \rar[two heads]{\cup^{\nu S}_1}
            \dar[two heads]{\cup^{\nu c}_0}
        & \pi_0\Diffp(S^2\tm D^2)
            \dar[tail,two heads]
        \\
        \pi_1\Emb_*(\nu S^2, S^4)
            \rar[tail]
        & \pi_0\Diffp(S^1\tm D^3)
            \rar[two heads]
        & \pi_0\Diff(S^4)
    \end{tikzcd}
    \]
    where $\pi_1\Emb_*(\nu S^2,S^4)\cong\pi_1\Embp(\nu D^2,D^4)$ is the subgroup of $\pi_1\Emb(\nu S^2,S^4)$ of index two, realized by the families of spheres in which one hemisphere is fixed throughout. 
\end{mainthm}

This diagram arises by putting together the long exact sequences of homotopy groups of fibrations~\eqref{eq-intro:fibrations}, so that $i_0^{\nu c}$ is the induced map on the kernels of the bottom right square. The connecting maps $\delta^e$ are induced from ambient isotopy extension, and showing their injectivity is the main task in the proof. This follows from the fact that both $c$ and $S$ are isotopic into the boundary of the ambient manifold. We refer to Section~\ref{subsec:setup} for details. 

Only one of the groups appearing in Theorem~\ref{mainT:diagram} has been computed in the literature so far: there is an isomorphism 
\begin{equation}\label{eq-intro:pi1S1D3} 
    \pi_1(\Emb(\nu S^1, S^1 \tm D^3); \nu c) 
    \cong \Z \tm \Z/2 \tm \Z^\infty, 
\end{equation} 
where $\Z$ corresponds to rotations of the circle and $\Z/2$ to rotations of the framings, and $\Z^\infty$ is generated by \emph{grasper families $\lambda^{c\mG_k}_\bull$ on $c$}, that are shown in Figure~\ref{F-intro}(ii). This result follows from computations by Budney and Gabai~\cite{BG} using embedding calculus, or alternatively by the fourth author and Teichner~\cite{KT-4dLBT,K-Dax,K-Dax2} using Dax’s study of \emph{metastable} homotopy groups of embedding spaces; see also Theorem~\ref{T:pi1-circles}. 

Crucial for all those computations is that the codimension in $\Emb(\nu S^1, S^1\tm D^3)$ is three, whereas the situation is more mysterious for the codimension two embedding space $\Emb(\nu S^2, S^2\tm D^2)$. However, we use Theorem~\ref{mainT:diagram} to show the following.
\begin{mainthm}
\label{mainT:infinite-subgroups}
\begin{itemize}
    \item
    There is an injective homomorphism of abelian groups
    \[
        trade\colon\pi_1(\Emb(\nu S^1,S^1\tm D^3);\nu c) \into \pi_1(\Emb(\nu S^2,S^2\tm D^2);\nu S)
    \]
    whose cokernel is isomorphic to $\pi_1\Embp(\nu D^2, D^4)$. 
    The subgroup $\Z \tm \Z/2 \tm \Z^\infty<\pi_1\Emb(\nu S^2,S^2\tm D^2)$ has $\Z$ generated by rotations of the framings and $\Z/2$ by rotations of the sphere, and $\Z^\infty$ by \emph{grasper families $\lambda^{S\mG_k}_\bull$ on $S$} from Figure~\ref{F-intro}(iii).

\item
    Moreover, there is an injective homomorphism of abelian groups
    \[
        \pi_1(\Emb(\nu S^1,S^1\tm D^3);\nu c) \into \pi_0\Diffp(S^1\tm S^2\tm I),
    \]
    whose cokernel is isomorphic to $\pi_0\Diffp(S^1\tm D^3)$.
    The subgroup $\Z \tm \Z/2 \tm \Z^\infty< \pi_0\Diffp(S^1\tm S^2\tm I)$ has $\Z$ generated by the Dehn twist on $S^1$ and $\Z/2$ by the Dehn twist on $S^2$, and $\Z^\infty$ by \emph{barbell diffeomorphisms} $\phi^{\mB_k}$ from Figure~\ref{F-intro}(i).
\end{itemize}
\end{mainthm}
\begin{figure}[!htbp]
\begin{minipage}{0.327\linewidth}
   \labellist                             % <- !
    \tiny\hair 2pt
    \pinlabel \blue{$k$} at 40 42
    \pinlabel \blue{$S_1$} at 130 52
    \pinlabel \blue{$S_2$} at 310 52
    \pinlabel $S$ at 10 140
    \pinlabel $c$ at 240 46
\endlabellist
    
   \centering
   \includegraphics[width=0.95\linewidth]{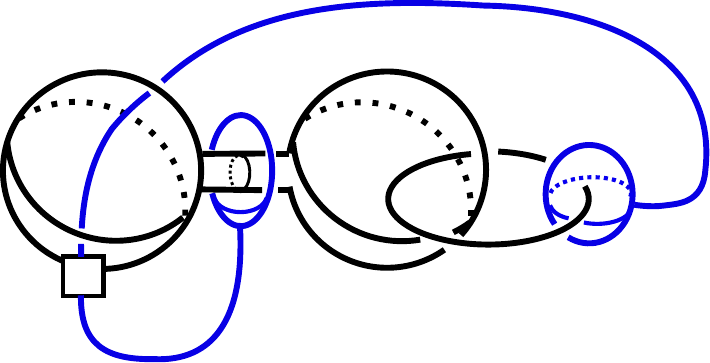}
\end{minipage}
\begin{minipage}{0.327\linewidth}
\labellist                             % <- !
    \tiny\hair 2pt
    \pinlabel \blue{$k$} at 40 42
    \pinlabel \blue{$S_1$} at 130 52
    \pinlabel \blue{$S_2$} at 310 52
    \pinlabel $S$ at 10 140
    \pinlabel $c$ at 240 46 
\endlabellist

   \centering
    \includegraphics[width=0.95\linewidth]{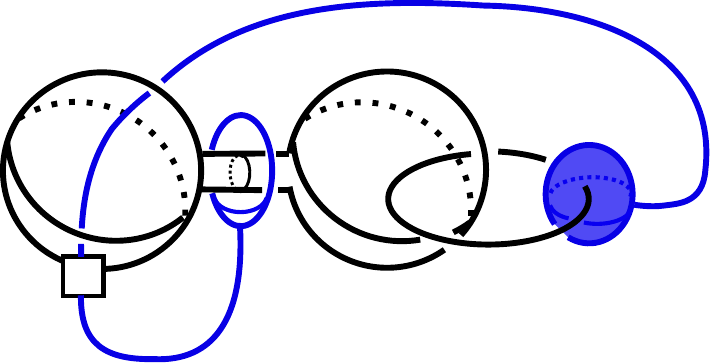}
\end{minipage}
\begin{minipage}{0.327\linewidth}
\labellist                             % <- !
    \tiny\hair 2pt
    \pinlabel \blue{$k$} at 40 42
    \pinlabel \blue{$S_1$} at 130 52
    \pinlabel \blue{$S_2$} at 310 52
    \pinlabel $S$ at 10 140
    \pinlabel $c$ at 240 46 
\endlabellist
  
   \centering
    \includegraphics[width=0.95\linewidth]{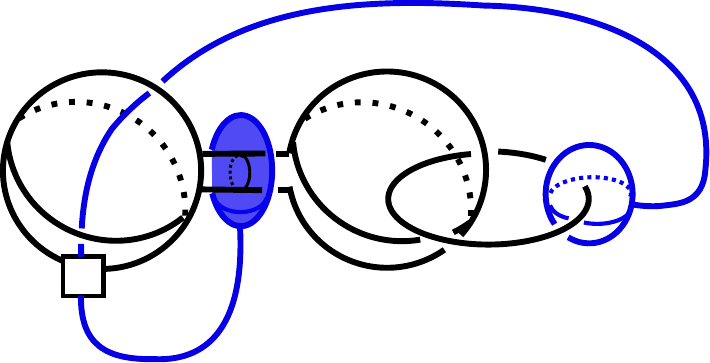}
\end{minipage}
\caption[The barbell, the grasper on a circle, the grasper on a sphere.]{
In each of three pictures we have a Hopf link $S\sqcup c\colon S^2\sqcup S^1\into S^4$ in black; its fixed thickening $\nu(S\sqcup c)$ is not drawn.
   
(i) The support of the diffeomorphism $\phi^{\mB_k}$ of $S^1\tm S^2\tm I=S^4\sm\nu (S\sqcup c)$ is the \emph{barbell} $\mB_k\subset S^1\tm S^2\tm I$. This is a thickening of the union of the blue spheres $S_1$ and $S_2$ and the bar connecting them (which links $k\geq1$ times with $S$).

(ii) The grasper $c\mG_k$ is the union of $\mB_k$ and the ball $Q_2$ bounded by $S_2$. The family $\lambda^{c\mG_k}_t\colon \nu S^1\into S^1\tm D^3$ drags the part $\nu c\cap Q_2$ along the bar until it touches $S_1$, then twirls it around $S_1$, and then goes back.

(iii) The grasper $S\mG_k$ is the union of $\mB_k$ and the ball $Q_1$ bounded by $S_1$. The family $\lambda^{S\mG_k}_t\colon\nu S^2\into S^2\tm D^2$ drags the part $\nu S\cap Q_1$ along the bar until it touches $S_2$, then twirls its around $S_2$, and then goes back.}
   \label{F-intro}
\end{figure}

The map $trade$ in Theorem~\ref{mainT:infinite-subgroups} is related to \emph{trading 1-2 to 2-3 pairs of $5$-dimensional handles} in pseudo-isotopies of $S^4$, as discussed by the second author in \cite{Gay}; we do not explicitly consider this aspect in the present paper.

Barbell diffeomorphisms were introduced by Budney and Gabai~\cite{BG}. The key property of $\phi^{\mB_k}$ is that its extension by the identity under the inclusion $S^1\tm S^2\tm I\subset S^1\tm D^3$ admits a null-isotopy. When applied to $\nu c$, this null-isotopy precisely recovers the grasper family $\lambda^{c\mG_k}_\bull$. Similarly, extending $\phi^{\mB_k}$ by the identity in $S^1\tm S^2\tm I\subset S^2\tm D^2$ yields the family $\lambda^{S\mG_k}_\bull$ when applied to $\nu S$.

Most of Theorem~\ref{mainT:infinite-subgroups} follows from Theorem~\ref{mainT:diagram} and \eqref{eq-intro:pi1S1D3}, as explained in Section~\ref{sec:diagram}. The additional part is the explicit description of the generators, given in Section~\ref{sec:classes}. However, in this approach to Theorem~\ref{mainT:infinite-subgroups}, the nontriviality of the $\Z^{\infty}$ factors in both the mapping class group of $S^1 \tm S^2 \tm I$ and the fundamental group of framed 2-spheres in $S^2 \tm D^2$ hinges on the presence of this factor in the group $\pi_1\Emb(\nu S^1, S^1 \tm D^3)$, by~\eqref{eq-intro:pi1S1D3}. 

This leads us to the second main result of this paper: we provide an alternative proof that these three groups are infinitely generated, without relying on previous metastable computations. Instead, our approach uses the work on pseudo-isotopies of Hatcher and Wagoner~\cite{HW}, as well as Singh ~\cite{Singh}, to demonstrate that barbell diffeomorphisms of $S^1 \tm S^2 \tm I$ represent distinct classes. That is, combining the following Theorem \ref{mainT:PI} with Theorem~\ref{mainT:diagram} gives another proof of Theorem~\ref{mainT:infinite-subgroups}, that does not use the computation~\eqref{eq-intro:pi1S1D3}.

\begin{mainthm}\label{mainT:PI}
    The barbell diffeomorphisms $\phi^{\mB_k}\in\pi_0\Diffp(S^1\tm S^2\tm I)$ for $k\geq1$ are nontrivial and pairwise distinct. They are pseudo-isotopic to the identity and detected by the Hatcher--Wagoner second obstruction $\Theta$.
\end{mainthm}

Previously, Singh~\cite{Singh} used $\Theta$ to show that $\pi_0\Diffp(S^1\tm S^2\tm I)$ contains an infinitely generated free abelian subgroup of classes that are pseudo-isotopic but not isotopic to the identity. Our Theorem \ref{mainT:PI} should be seen as a refinement of Singh's result that gives explicit representatives of these classes.

\subsection*{Outline}

We prove Theorem~\ref{mainT:diagram} in Section~\ref{sec:diagram} relying on Cerf--Palais fibrations~\eqref{eq-intro:fibrations} and the previous work of the last author with Peter Teichner~\cite{KT-4dLBT,KT-mcg}.

In Section~\ref{sec:classes} we prove Theorem~\ref{mainT:infinite-subgroups} by defining $\phi^{\mB_k}$, $\lambda^{c\mG_k}_\bull$, $\lambda^{S\mG_k}_\bull$, and observing relations between them. We define the model barbell $\sB$ as the complement in $D^4$ of two disjoint framed neat arcs $\alpha^1\sqcup\alpha^2$, and the model grasper $\sG$ as the complement of one of them. Thus, $\sB\subset\sG$ in two different ways. A barbell diffeomorphism has support in the image of a barbell $\mB\colon\sB\into X$, where it agrees with the barbell map $b\colon\sB\to\sB$. The map $b$ is the restriction of is a map $b_1\colon\sG\to\sG$ that preserves $\alpha^1$. This in turn is the endpoint of an ambient isotopy $b_\bull$ that extends the loop of arcs $\alpha^1_\bull\colon D^1 \into \sG$, known as the \emph{twirling family}. It is exactly the twirling family that gives grasper families of circles $\lambda^{c\mG_k}_\bull$, and its key feature is that the movement of $c$ can be traded for the movement of $S$, giving $\lambda^{S\mG_k}_\bull$.

Turning now to the proof of Theorem~\ref{mainT:PI} in Section~\ref{sec:PI}, we fix $k\geq1$ and consider the barbell diffeomorphism $\phi^{\mB_k}$. The first goal is to construct an explicit family
\[
    R_t\colon S^2\into S^1\tm S^2\tm I \# S^2\tm S^2
\]
with $R_0=R_1=R=S^2\tm \{p\}$, whose trace determines a pseudo-isotopy $\Phi$ from $\phi^{\mB_k}$ to the identity. This correspondence between the loop $R_\bull$ and the pseudo-isotopy $\Phi$ is derived from Cerf's work and will not be described here (see \cite{cerf1970,HW,Gay,gabai2024,GGHKP} for more details). Once the loop $R_\bull$ is constructed, we rely on Singh's work~\cite{Singh} to compute the Hatcher--Wagoner invariant $\Theta(\Phi)$, by looking at the intersection locus of $R_\bull$ with the fixed sphere $G = \{p\}\tm S^2$.

In order to define $R_\bull$ we first construct in Section~\ref{subsec:BB-to-S2} a path $\wt{\phi}_\bull$ of diffeomorphisms of $S^1\tm S^2\tm I \# S^2\tm S^2$ that starts at the identity $\wt{\phi}_0=\id$ and ends at a certain diffeomorphism $\wt{\phi}_1=\wt{\phi}$. What is important is that $\wt{\phi}$ is isotopic to $\phi^{\mB_k}$ after performing surgery to $S^2\tm \{p\}\subset S^1\tm S^2\tm I \# S^2\tm S^2$; in other words, $\wt{\phi}_t$ is a stable isotopy for $\phi^{\mB_k}$ (see Section~\ref{subsec:BB-to-double-BB}). From this path we extract the desired family of 2-spheres by simply defining $R_t\coloneqq\wt{\phi}_t(R)$. We compute its Hatcher--Wagoner obstruction in Section~\ref{subsec:compute}, after briefly recalling the work of Hatcher, Wagoner, and Singh in Section~\ref{subsec:HWS}.

\subsection*{Acknowledgments}
DH and DK are thankful to ETH Zurich for excellent working conditions and generous support for DH's visit.

%%%%%%%%%%%%%%%%%%%%%%
\section{The diagram}\label{sec:diagram}

\subsection{The setup}\label{subsec:setup}
Let $S$ denote an unknotted 2-sphere in $S^4$ and $c$ its meridian circle as in Figure~\ref{F-intro}, so that the pair $S\sqcup c$ forms a standard Hopf pair in $S^4$. We fix a tubular neighborhood $\nu(S\sqcup c)$ and the identifications
\[
    S^4\sm \nu(c\sqcup S) \dashrightarrow S^1\tm S^2\tm I,\quad
    S^4\sm \nu c \dashrightarrow S^2\tm D^2,\quad
    S^4 \sm \nu S \dashrightarrow S^1\tm D^3,
\] 
with $I\coloneqq[0,1]$, in such a way that the following cube commutes:
\[
\begin{tikzcd}[row sep=1.3em, column sep = 1.5em,nodes={scale=0.93}]
    S^4\sm \nu(c\sqcup S) 
        \arrow[rr, dashed] \arrow[dr, swap,"\cup\nu S"] 
        \arrow[dd,swap, "\cup\nu c"] 
    && S^1\tm S^2\tm I 
        \arrow[dd, "\cup_0{\nu c}", near end] 
        \arrow[dr,"\cup_1{\nu S}"] 
    \\
    & S^4\sm \nu c 
        \arrow[rr,dashed]
        \arrow[dd,swap,near start,"\cup\nu c"] 
    && S^2\tm D^2 
        \arrow[dd,"\cup_{\partial}{\nu c}"] 
    \\
    S^4\sm \nu S 
        \arrow[rr,dashed] 
        \arrow[dr, "\cup\nu S"] 
    && S^1\tm D^3 
        \arrow[dr,"\cup_{\partial}{\nu S}"] 
    \\
    & S^4 
        \arrow[rr,dashed,"\id"] 
    && S^4
    \end{tikzcd}
\]
The maps $\cup\nu c$ and $\cup\nu S$ are the natural inclusions, whereas the maps on the right are induced by those inclusions under the fixed identifications. Extending diffeomorphisms by the identity on these subspaces gives the commutative diagram of spaces:
\begin{equation}\label{eq:diagram-of-spaces}
    \begin{tikzcd}[nodes={scale=0.93}]
        \Diffp(S^1\tm S^2\tm I)
            \rar{\cup_1^{\nu S}}\dar{\cup_0^{\nu c}}
        & \Diffp(S^2\tm D^2)
            \rar{\res_{\nu S}}\dar{\cup_\partial^{\nu c}}
        & \Emb(\nu S^2, S^2\tm D^2)\dar\\
        \Diffp(S^1\tm D^3)
            \rar{\cup_\partial^{\nu S}}\dar{\res_{\nu c}}
        & \Diff(S^4)
            \rar{\res^{S^4}_{\nu S}}\dar{\res^{S^4}_{\nu c}}
        & \Emb(\nu S^2,S^4)\\
        \Emb(\nu S^1,S^1\tm D^3)\rar
        & \Emb(\nu S^1,S^4)
    \end{tikzcd}
\end{equation}
The maps $\res$ restrict to appropriate subspaces $\nu S\cong S^2\tm D^2$ or $\nu c\cong S^1\tm D^3$, and are fibrations by Cerf~\cite{Cerf-plongements} and Palais~\cite{Palais}. The fibers consist of diffeomorphisms that are the identity on that subspace, so each column and each row in this diagram is a fibration sequence. 
Therefore, we can form a diagram that intertwines the corresponding four long exact sequences of homotopy groups. The following theorem says that the bottom parts of these long exact sequences reduce to trivial short exact sequences.
\begin{thm}\label{T:main-diagram}
    There is a commutative diagram of abelian groups
    \[
    \begin{tikzcd}[nodes={scale=0.93}]
        & \pi_1\Embp(\nu S^1,S^1\tm D^3)
            \dar[tail,shift left]{\delta_{\nu c}^{S^1\tm D^3}}
            \rar{i_0}
        & \pi_1\Emb_*(\nu S^1,S^4)
            \dar[tail,shift left]{\delta_{\nu c}^{S^4}}
        \\ 
        \pi_1\Embp(\nu S^2,S^2\tm D^2)
            \rar[tail,shift left]{\delta_{\nu S}^{S^2\tm D^2}}
            \dar{i_1}
        & \pi_0\Diffp(S^1\tm S^2\tm I)
            \rar[two heads,shift left]{\cup_0^{\nu S}}
            \dar[two heads,shift left]{\cup_1^{\nu c}}
        & \pi_0\Diffp(S^2\tm D^2)
            \dar[two heads,shift left]{\cup_\partial^{\nu c}}
        \\
        \pi_1\Emb_*(\nu S^2, S^4)
            \rar[tail,shift left]{\delta_{\nu S}^{S^4}} 
        & \pi_0\Diffp(S^1\tm D^3)
            \rar[two heads,shift left]{\cup_\partial^{\nu S}}
        & \pi_0\Diff^+(S^4)
    \end{tikzcd}
    \]
    with each row and each column a trivial group extension. For both $k=1,2$ we have $\pi_1\Emb(\nu S^k,S^4)\cong\Z/2\tm\pi_1\Emb_*(\nu S^k,S^4)$ and $\pi_1\Emb_*(\nu S^k,S^4)\cong\pi_1\Embp(\nu D^k,D^4)$.
\end{thm}
% \begin{figure}[!htbp]
%     \centering
%      \includegraphics[width=.13\linewidth]{Images/Hopf-link.pdf}
%     \caption{The Hopf pair $S\sqcup c\colon S^2\sqcup S^1\into S^4$.}
%     \label{F:Hopf}
% \end{figure}

\subsection{The proof of Theorem~\ref{T:main-diagram}}\label{subsec:proof}
We will use the following result of the last author and Peter Teichner.
\begin{thm}[{\cite{KT-4dLBT}, \cite{KT-mcg}}]\label{T:LBT}
    Whenever $K\colon\Sigma^k\into X$ is isotopic into the boundary of $X$, then there is an explicit split group extension:
    \[\begin{tikzcd}
        \pi_1\Embp(\nu \Sigma,X);\nu K)
            \rar[tail]{\delta_{\nu K}}
        & \pi_0\Diffp(X\sm\nu K)
            \rar[two heads,shift left]{\cup^{\nu K}}
        & \pi_0\Diffp(X),
            \lar[tail,shift left]{\sigma_{\nu K}}
    \end{tikzcd}
    \]
\end{thm}
    The version where $\Sigma=D^k$ and $K$ is neat is the Light Bulb Trick of \cite{KT-4dLBT}.
    The version for general submanifolds will appear in~\cite{KT-mcg}. Note that in general these group extensions will not be trivial. See also \cite[Lem.5.15]{KT-4dLBT}. 

We will also need the following facts, known to the experts. Note that any manifold of the shape $Y\tm I$ admits a stacking operation $Y\tm I\cup_{Y\tm\{1\}=Y\tm\{0\}} Y\tm I\cong Y\tm I$. Thus, given two diffeomorphisms of $Y\tm I$ we can stack them next to each other.

\begin{lem}
    For any manifold $Y$ the standard group structure on $\pi_0\Diffp( Y\tm I)$ is equivalent to the one obtained by stacking of diffeomorphisms. Moreover, this group is abelian.
\end{lem}
\begin{proof}
    Since diffeomorphisms $\varphi_1,\varphi_2\in\Diffp(Y\tm I)$ are the identity on a neighbourhood of the boundary, we can isotope $\varphi_2$ to be supported on $Y\tm [0,1/2]$ and $\varphi_1$ in $Y\tm[1/2,1]$. In other words, there are isotopies from $\varphi_2$ to $\varphi_2\otimes\Id$ and from $\varphi_1$ to $\Id\otimes\varphi_1$. Then $\varphi_2\circ\varphi_1$ is isotopic to $(\varphi_2\otimes\Id)\circ(\Id\otimes\varphi_1)$. These stacking operation and composition are clearly distributive, so we have
    \[
        \varphi_2\circ\varphi_1=(\varphi_2\circ\Id)\otimes(\Id\circ\varphi_1)=\varphi_2\otimes\varphi_1.
    \]
    To see that $\pi_0\Diffp(Y\tm I)$ is abelian note that we could have isotoped the supports the other way. In other words, this is an Eckmann-Hilton argument:
    \[
        \varphi_2\circ\varphi_1=
        (\Id\otimes\varphi_2)\circ(\varphi_1\otimes\Id)=
        (\Id\circ\varphi_1)\otimes(\varphi_2\circ\Id)
        =\varphi_1\otimes\varphi_2=\varphi_1\circ\varphi_2.\qedhere
    \]
\end{proof}

We are now ready to prove Theorem~\ref{T:main-diagram}.

\begin{proof}[Proof of Theorem~\ref{T:main-diagram}]
    The diagram is obtained by applying long exact sequences in homotopy groups to the diagram of spaces~\eqref{eq:diagram-of-spaces}. It follows from Theorem~\ref{T:LBT} that the middle row and middle column form semi-direct products. However, all the mapping class groups in the diagram are abelian by the last lemma, so they are in fact just products!

    For the rightmost column and the bottom row the target manifold is now closed, namely $S^4$, so we cannot use Theorem~\ref{T:LBT}. However, the inclusion $w\colon D^4\into S^4$ of the western hemisphere induces an isomorphism $\pi_0\Diffp(D^4)\cong\pi_0\Diff^+(S^4)$, and Theorem~\ref{T:LBT} does give the corresponding sequence in the 4-ball. Therefore, for $k=1$ or $k=2$ we can form the commutative diagram
\[\begin{tikzcd}[nodes={scale=0.91}]
        \pi_2V_4(S^4)
            \dar\rar[equals]
        & \pi_2V_4(S^4)=\pi_2SO_4=0
            \dar &&\\
        \pi_1\Diffp(D^4)
            \rar{0}\dar
        & \pi_1(\Embp(\nu D^{3-k},D^4);\nu U)
            \rar[tail]{\delta_{\nu U}}
            \dar{w(-)\cup e(U)}
        & \pi_0\Diffp(S^k\tm D^{4-k})
            \rar[two heads,shift left]{\cup h^{k+1}}
            \dar[equals]
        & \pi_0\Diffp(D^4)
            \lar[tail,shift left]{\sigma_{\nu U}}
            \dar{\cong}[swap]{w}\\
        \pi_1\Diff^+(S^4)
            \rar\dar
        & \pi_1(\Emb(\nu S^{3-k},S^4);\nu S)
            \rar{\delta_{\nu S}}
            \dar
        & \pi_0\Diffp(S^k\tm D^{4-k})
            \rar[two heads]{\cup^{\nu S}}
        & \pi_0\Diff^+(S^4)\\
        \pi_1 V_4(S^4)
            \rar[equals]
        & \pi_1 V_4(S^4)
        &&
    \end{tikzcd}
\]
    Here $\cup h^{k+1}$ extends by the identity over a $(k+1)$-handle, and has a section by Theorem~\ref{T:LBT}. We fix a basepoint $U\colon D^{3-k}\into D^4$ and the corresponding basepoint $S=w(U)\cup e(U)\subset w(D^4)\cup e(D^4)=S^4$. The diagram commutes because if a diffeomorphism of $S^4$ fixes the eastern hemisphere $e(D^4)$, then its restriction to $S\colon S^2\into S^4$ fixes $S(S^2)\cap e(D^4)$, so comes from $D^2\into D^4$ by applying $w$ and capping with a fixed disk $e(U)$.
    
    Let us observe that $V_4(S^4)\cong SO_5$, since the unit vector that determines a point $p\in S^4$ completes a normal 4-frame at $p$ to a 5-frame in the ambient $\R^5$. The map $\pi_1\Diff^+(S^4)\to\pi_1V_4(S^4)\cong\pi_1SO_5\cong\Z/2$, that takes the derivative at the basepoint $e\in S^4$, has a section: send $1\in\Z/2$ to the rotation of $S^4$ around the axis through $e$.
    
    By the commutativity of the diagram we conclude that there are group isomorphisms
        \begin{equation}\label{eq:sphere-subgroup}
            \pi_1(\Emb(\nu S^k,S^4);\nu S)\cong\Z/2\tm\pi_1(\Embp(\nu D^k,D^4);\nu U),
        \end{equation}
    with $\Z/2$ hit by $\pi_1\Diff^+(S^4)$, whereas the subgroup $\pi_1(\Emb_*(\nu S^k,S^4);\nu S)$ of spheres with a fixed hemisphere is isomorphic to $\pi_1(\Embp(\nu D^k,D^4);\nu U)$ and survives under the connecting map $\delta_{\nu S}$ to $\pi_0\Diffp(S^k\tm D^{4-k})$. This concludes the proof of the theorem.
\end{proof}

\subsection{Consequences}
The groups in the leftmost column of Theorem~\ref{T:main-diagram} are fundamental groups of codimension two embedding spaces, and thus hard to compute.
On the other hand, the groups in the top row for codimension three embeddings, and are known:
\begin{align*}
    \pi_1(\Embp(\nu S^1,S^1\tm D^3),\nu c) &\cong \Z\tm \Z/2\tm \Z^\infty, \\
    \pi_1(\Emb_*(\nu S^1,S^4),\nu c) &\cong \pi_1(\Embp(\nu D^1,D^4);\nu c)\cong 1.
\end{align*}
For the former see \eqref{eq-intro:pi1S1D3} above or Theorem~\ref{T:pi1-circles} below.
Using the latter and \eqref{eq:sphere-subgroup} we find
\[
    \pi_1(\Embp(\nu S^1,S^4);\nu c)\cong \Z/2.
\]
Combining all of this finishes the proof of Theorem~\ref{mainT:diagram}. 
We leave it as an exercise for the reader to give a slightly different proof of Theorem~\ref{mainT:diagram}, using the neat case of Theorem~\ref{T:LBT} and the fibration $\res\colon\Diffp(S^1\tm D^3)\to \Embp(D^3,S^1\tm D^3)$ in place of $\res^{S^4}_{\nu c}$. The proof given above has the advantage in that it shows that we can move the $\Z/2$ from the kernel down to $\pi_1\Emb(\nu S^2, S^4)$, giving the following diagram.
\begin{cor}\label{cor:main}
    There is an exact diagram of abelian groups, with all injections split:
    \[
    \begin{tikzcd}[nodes={scale=0.93}]
        &&&\Z/2
            \dar[tail,two heads]\\
        & \Z^\infty\tm\Z
            \rar[tail,two heads]
            \dar[tail]
        & \pi_1\Embp(\nu S^1,S^1\tm D^3)
            \dar[tail]
            \rar{}
        & \pi_1\Emb(\nu S^1,S^4)
            \dar{0}
        \\ 
        &\pi_1\Embp(\nu S^2,S^2\tm D^2)
            \rar[tail]
            \dar[two heads]
        & \pi_0\Diffp(S^1\tm S^2\tm I)
            \rar[two heads]
            \dar[two heads]
        & \pi_0\Diffp(S^2\tm D^2)
            \dar[tail,two heads]
        \\
        \Z/2
            \rar[tail]
        & \pi_1\Emb(\nu S^2, S^4)
            \rar 
        & \pi_0\Diffp(S^1\tm D^3)
            \rar[two heads]
        & \pi_0\Diff^+(S^4)
    \end{tikzcd}
    \]
\end{cor}
In particular, from Theorem~\ref{mainT:diagram} we obtain the previously known isomorphisms
\begin{align*}
    \pi_0\Diffp(S^2\tm D^2) 
            &\cong\pi_0\Diff(S^4)\cong\pi_0\Diffp(D^4),\\
        \pi_0\Diffp(S^1\tm D^3) 
            &\cong \pi_1\Emb_*(\nu S^2,S^4)\tm \pi_0\Diff(S^4),
            %\cong \pi_1\Embp(\nu D^2,D^4)\tm \pi_0\Diffp(D^4),
\end{align*}
whereas combining Theorem~\ref{mainT:diagram} with the computation~\eqref{eq-intro:pi1S1D3} implies some new ones:
\begin{align*}
    \pi_0\Diffp(S^1\tm S^2\tm I)
        & \cong \pi_1\Emb(\nu S^2,S^2\tm D^2)\tm\pi_0\Diffp(S^2\tm D^2)\notag\\
        & \cong \Z\tm\Z/2\tm\Z^\infty\tm \pi_0\Diffp(S^1\tm D^3),\\
    \pi_1(\Emb(\nu S^2,S^2\tm D^2);\nu S)
        & \cong \Z\tm\Z/2\tm\Z^\infty\tm\pi_1\Embp(\nu D^2, D^4)\notag\\
        & \cong \Z\tm\Z^\infty\tm\pi_1\Emb(\nu S^2, S^4).
\end{align*}
In order to finish the proof of Theorem~\ref{mainT:infinite-subgroups} it remains to describe the generators corresponding to the $\Z^\infty$ factors, which is done in the next section (see Section~\ref{subsec:identify}).

%%%%%%%%%%%%%%%%%%%%%%
\section{The classes}\label{sec:classes}

In this section we define our barbell and grasper classes, relying on the previous work of many authors~\cite{BG,Gay,GH,KT-4dLBT,K-graspers-mcg}. In Section~\ref{subsec:models} we define the model barbell $\sB$ and the barbell map $b\colon\sB\to\sB$, the model grasper $\sG$ and the twirling family $\alpha_\bull\colon\nu D^1\into\sG$, and describe the effect of $b$ on the midball $\Delta\subseteq\sB$. In Section~\ref{subsec:barbells-graspers} we define barbells $\mB$, barbell diffeomorphisms $\phi^\mB$, and grasper families $\lambda^{c\mG_k}_\bull$ and $\lambda^{S\mG_k}_\bull$. Moreover, Theorem~\ref{T:pi1-circles} explains~\eqref{eq-intro:pi1S1D3}. Finally, in Section~\ref{subsec:identify} we finish the proof of Theorem~\ref{mainT:infinite-subgroups}.

\subsection{Models}\label{subsec:models}
% For $Y\subset\R^3$ and $\epsilon\geq0$ denote $Y_{\leq \epsilon} \coloneqq \{x\in \R^3: \operatorname{dist}(x,Y)\leq \epsilon\}$. In particular, $(x_1,x_2,x_3)_{\leq r}\subset \R^3$ is the embedded 3-ball of radius $r$ centered at $(x_1,x_2,x_3)\in \R^3$. 
\begin{defn}[Model barbell, model grasper]\label{D:models}
\hfill
\begin{itemize}
\item 
    Fix a framed embedding $\nu(\alpha^1\sqcup\alpha^2)\colon \nu(D^1\sqcup D^1)\into D^4$ and define the \emph{model barbell}
    \[
        \sB\coloneqq D^4\sm\nu(\alpha^1\sqcup\alpha^2).
    \]
    For concreteness we parametrize this as follows. Let $D^4=E\tm D^1$ for the ellipsoid $E\subset\R^3$ as in Figure~\ref{F:models} and the time axis $s\in D^1$. Let $\nu\alpha^i=Q_i\tm D^1$ for the open unit balls $Q_i$ centered at $(\pm2,0,0)\in E$ as in the picture, $i=1,2$. Then $\sB=(E\sm (Q_1\sqcup Q_2))\tm D^1$ is the union for $s\in D^1$ of the spaces
    \[
    \sB_s=(E\sm (Q_1\sqcup Q_2))\tm\{s\}.
    \]
\item 
    The spheres $S_i=\partial \ol{Q}_i$ are called the \emph{cuffs}, and the arc $\{(x,0,0): -1\leq x \leq 1\}$ connecting them the \emph{bar}. 
    Thus, $\sB$ is a 4-dimensional thickening of the union of the cuffs and the bar (colored in blue in Figure~\ref{F:models}).
\begin{figure}[!htbp]
    \centering
    \includegraphics[width=0.3\linewidth]{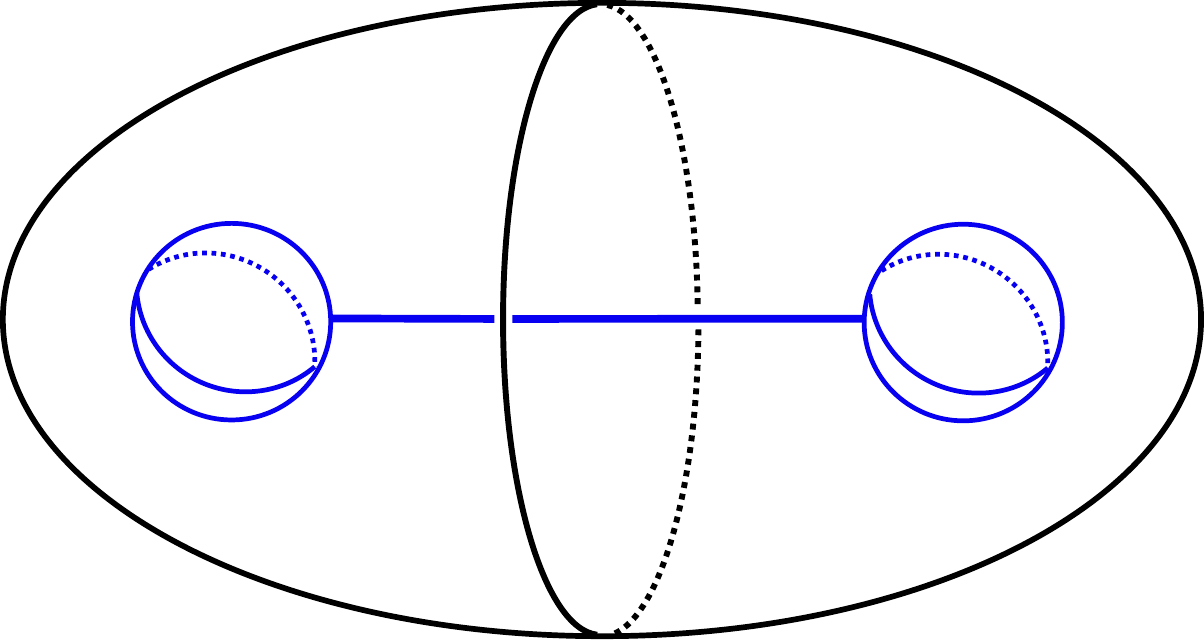}
    \caption{The subspace $\sB_0\subset\sB$ contains the cuffs and the bar in blue. The model barbell $\sB$ is the product of this space with $D^1$.}
    \label{F:models}
\end{figure}
\item
    Define the \emph{midball} $\Delta\subset\sB$ as the 3-ball $\Delta\coloneqq \sB\cap\{x=0\}$. Equivalently, $\Delta=\bigcup_{s\in D^1}\Delta_s$ for the vertical disks $\Delta_s=\sB_s\cap\{x=0\}$, as in the top row of Figure~\ref{F:midball}.
\item 
    Define the \emph{model grasper} $\sG$ as the union of the model barbell $\sB$ and the ball $Q_1$ bounded by the cuff $S_1$. We denote the remaining cuff by $L\coloneqq S_2$ and call it the \emph{leaf}. Thus, $\sG=(E\sm Q_2)\tm D^1$ and there is a canonical inclusion $fill_1\colon\sB\into\sG$. 
\item 
    We also have an inclusion $fill_2\colon\sB\into\sG$ by filling in the ball $Q_2$ of the cuff $S_2$, and then isotoping the remaining cuff $S_1$ to agree with the leaf in $\sG$.
\qedhere
\end{itemize}
\end{defn}
\begin{defn}[Twirling family, barbell map]\label{D:twirl}
\hfill
\begin{itemize}
\item
    We define an isotopy of $\nu \alpha^1\colon \nu D^1\cong Q_1\tm D^1\into \sG$ called the \emph{twirling family}
    \[
        \nu \alpha^1_t\colon \nu D^1\into \sG,
    \]
    as a thickening of the following isotopy of the arc $\alpha^1$ rel boundary. As in Figure~\ref{F:twirl}, we first drag the arc along the bar to reach a neighborhood of the sphere $L$, then twirl it around a parallel copy of $L$ using its foliation by meridians, and then return the arc along the bar back where it started, $\alpha^1_0=\alpha^1_1=\alpha^1$.
\begin{figure}[!htbp]
    \centering
     \includegraphics[width=.7\linewidth]{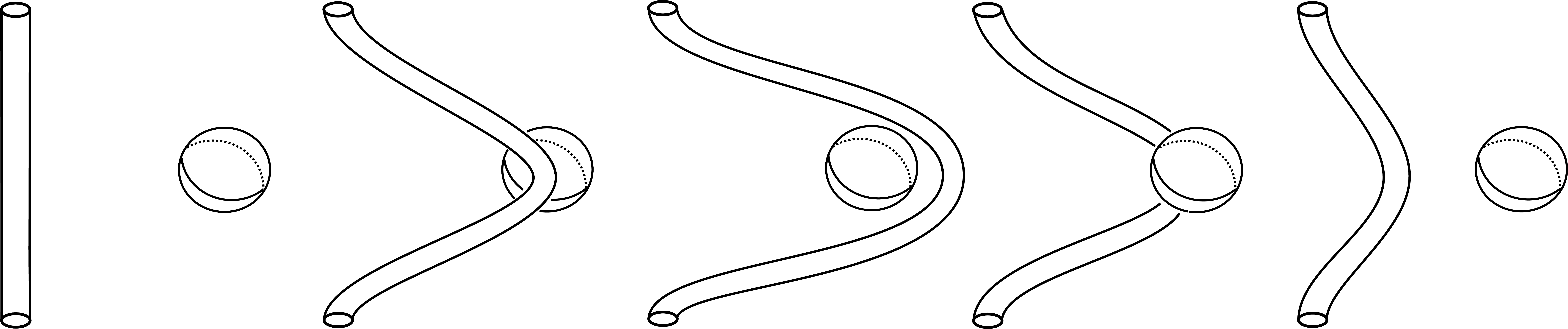}
    \caption{The twirling family.}
    \label{F:twirl}
\end{figure}
\item
    Let $b_\bull\colon\sG\to\sG$ be the path of diffeomorphisms of $\sG$ given by the ambient isotopy extension of $\nu \alpha^1_\bull$, starting with $b_0=\id$. Each $b_t$ is the identity near boundary of $\sB$. 
\item 
    For $t=1$ we have $b_1(\nu \alpha^1)=\nu \alpha^1$, so the restriction of $b_1$ to $\sB\subset\sG$ is a diffeomorphism $b\colon \sB\to \sB$ rel boundary.
    We call it the \emph{barbell map} $b\in\Diffp(\sB)$.
\qedhere
\end{itemize}
\end{defn}
The twirling $\alpha^1_t$ of the arc $\alpha^1$ around $L$ can be viewed as a loop $\alpha^1_t\sqcup \alpha^2\colon D^1\sqcup D^1\into D^4$ of string links, where $\alpha^1$ twirls around the meridian sphere $S_2$ of $\alpha^2$. The following is then not hard to see~\cite{BG,GH,K-graspers-mcg} and will be used in Section~\ref{subsec:identify}.
\begin{lem}\label{L:twirl}
    The loop $\alpha^1_t\sqcup\alpha^2$ is isotopic to the loop $\alpha^1\sqcup \alpha^2_t\colon D^1\sqcup D^1\into D^4$ of string links, where $\alpha^2_0=\alpha^2_1=\alpha^2$ twirls around the oppositely oriented meridian $-S_1$ of $\alpha^1$.
\end{lem}

In Section~\ref{subsec:BB-to-S2} we will make use of the 3-ball $b(\Delta)$, the image of the midball $\Delta\subset\sB$ under the barbell map $b$. The bottom row of Figure~\ref{F:midball} depicts the cross-sections $b(\Delta_s)$ for several moments $s\in D^1$. From the second moment to the fourth use the twirling isotopy of the thin orange tube around $S_2$. From the second moment to the first, and from the fourth to the last, use the following isotopy: take the small circle where the tube starts on the vertical disk, and drag it within the orange disk until it reaches the westernmost point on the outer orange blob, and then cancel the two blobs against each other. 

\begin{figure}[!htbp]
    \centering
    \includegraphics[width=\linewidth]{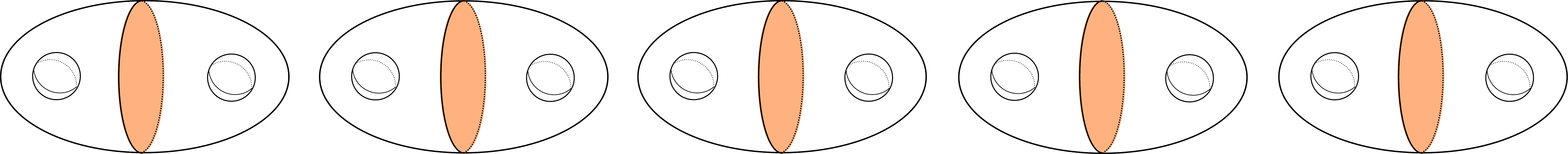}
    \vspace{1pt}
    
    \includegraphics[width=\linewidth]{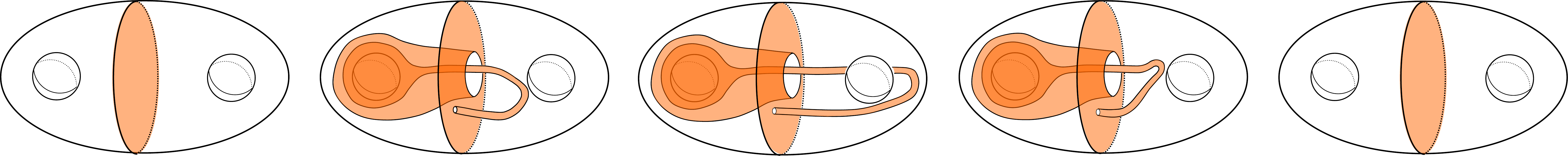}
    \caption{The top row depicts the cross-sections $\Delta_s$ of the midball $\Delta\subset\sB$, and the bottom row of its image $b(\Delta)\subset\sB$ under the barbell map.
    }
    \label{F:midball}
\end{figure}
To see that this is a correct description of $b(\Delta)$, let us look at the top row of Figure~\ref{F:midball} depicting the slices $\Delta_s$. The twirling isotopy $\alpha^1_\bull$ of the arc $\alpha^1$ is visible as the movement of the ball $Q_1$ around the sphere $S_2$. To reach a neighborhood of $S_2$ the ball first has to cross the orange disk $\Delta_s$ -- this is what creates the inner blob on that disk. As the ball goes around $S_2$ this blob gets stretched, forming a tube, and as the ball returns it crosses the orange disk again, creating the outer blob. For this description of $b(\Delta)$ see~\cite[Prop.6.1,Fig.36]{BG}; the only difference is that we have ``isotoped the neck (tube) out''.

\subsection{Barbell diffeomorphisms and grasper families}\label{subsec:barbells-graspers}

\begin{defn}\label{D: barbell-diffeo}
Fix a manifold $M$.
\begin{itemize}
    \item
    A \emph{barbell} is a smooth embedding $\mB\colon\sB\into M$ of the model barbell away from $\partial M$.
    \item
    Given a barbell $\mB\colon\sB \into M$, the \emph{barbell diffeomorphism} $\phi^\mB\colon M\to M$ is defined as 
    \[
    \phi^\mB (x) = 
    \begin{cases}
        x, & x \in M\sm \mB(\sB),\\
        \mB\circ b \circ \mB^{-1}(x), & x\in\mB(\sB),
    \end{cases}
    \]
    using the barbell map $b\in\Diffp(\sB)$ from Definition~\ref{D:twirl}.
\qedhere
\end{itemize}
\end{defn}

\begin{lem}\label{L:b-nullistpy}
    If $\mB\colon\sB\into M$ extends to a map $\mG\colon\sG\into M$ so that $\mB=\mG\circ fill_1$, then $\phi^\mB$ is isotopic to the identity via the ambient isotopy
    \[
    \phi^\mB_t (x) = 
    \begin{cases}
        x, & x \in M\sm \mG(\sG),\\
        \mG\circ b_t \circ \mG^{-1}(x), & x\in\mG(\sG).
    \end{cases}
    \]
\end{lem}
\begin{proof}
    We clearly have that $\phi^\mB_t$ is well defined, and $\phi^\mB_0=\Id$ and $\phi^\mB_1=\phi^\mB$ by definition. 
\end{proof}

\begin{rem}
    In \cite{BG} the space $\sB$ is called the barbell, $\mB$ is a barbell implantation, and twirling is spinning. In \cite{KT-4dLBT,K-Dax,K-graspers-mcg} twirling is described as swinging. We opted for this new distinctive name.
\end{rem}

Let us now return to the setting of Sections~\ref{sec:intro} and~\ref{sec:diagram}: we have a framed Hopf link $S\sqcup c\colon S^2\sqcup S^1\into S^4$, whose complement we identify with $S^1\tm S^2\tm I$. We also fix a barbell $\mB_k\colon \sB\into S^4\sm\nu(S\sqcup c)\cong S^1\tm S^2\tm I$ as in Figure~\ref{F-intro}(i).
\begin{defn}\label{D:graspers-on-c}
\hfill 
\begin{itemize}
\item 
    Define the grasper $c\mG_k\colon\sG\into S^4\sm\nu S\cong S^1\tm D^3$ as in Figure \ref{F-intro}(ii), by filling in the cuff $S_2$ of the barbell $\mB_k$. Up to a re-parametrization we can assume that $J\subset S^1$ is an interval such that $\nu c\cap c\mG_k=\nu c(\nu J) = c\mG_k\circ\nu\alpha^1(\nu D^1)$.
\item 
    The \emph{$c\mG_k$-grasper family} is the loop of framed circles in $S^1\tm D^3$ given for $t\in[0,1]$ by
    \[
    \lambda^{c\mG_k}_t\colon\nu S^1\into S^1\tm D^3,\quad \lambda^{c\mG_k}_t(z) \coloneqq\begin{cases}
        \nu c(z),                         & z\in \nu(S^1\sm J),\\
        c\mG_k\circ \nu \alpha^1_t\circ (\nu\alpha^1)^{-1}( c\mG_k^{-1}\nu c(z)),    & z\in \nu J.
    \end{cases}
    \]
    Since $\nu\alpha^1_0=\nu\alpha^1_1=\nu c|_{\nu J}$ this defines a class $\lambda^{c\mG_k}_\bull\in\pi_1(\Emb(\nu S^1, S^1\tm D^3);\nu c)$.
\item 
    Define the grasper $S\mG_k$ on $S:S^2\into S^2\tm D^2$ as in Figure \ref{F-intro}(iii), by filling in the cuff $S_1$ of the barbell $\mB_k$. Up to a re-parametrization we can assume that $J'\subset S^2$ is an annulus such that $\nu S\cap S\mG_k=\nu S(\nu J')= S\mG_k\circ\nu\alpha^2(\nu J'')$ where $J''=D^1\tm S^1\subset D^1\tm D^3=\nu D^1$.
\item 
    The \emph{$S\mG_k$-grasper family} is the loop of framed spheres in $S^2\tm D^2$ given by
    \[
    \lambda^{S\mG_k}_t\colon\nu S^2\into S^2\tm D^2,\quad
    \lambda^{S\mG_k}_t(z)\coloneqq\begin{cases}
        \nu S(z),          & z\in \nu(S^2\sm J'),\\
        S\mG_k\circ \nu\alpha^2_t\circ (\nu\alpha^2)^{-1} (S\mG_k^{-1}\nu S(z)),    & z\in \nu J'.
    \end{cases}
    \]
    Since $\nu\alpha^2_0|_{\nu J''}=\nu\alpha^2_1|_{\nu J''}=\nu S|_{\nu J'}$ this is a class $\lambda^{S\mG_k}_\bull\in\pi_1(\Emb(\nu S^2, S^2\tm D^2);\nu S)$.
\qedhere
\end{itemize}
\end{defn}

We can now describe the isomorphism mentioned in~\eqref{eq-intro:pi1S1D3}.
\begin{thm}[\cites{BG,K-Dax,K-Dax2}]\label{T:pi1-circles}
    There is an isomorphism
    \[
        \pi_1(\Emb(\nu S^1, S^1\tm D^3),\nu c)\cong \Z^{\infty}\tm \Z/2 \tm \Z,
    \]
    where the generator of $\Z$ is the rotation of the circle in place, of $\Z/2$ is the rotation of the normal frames at all points simultaneously, and of $\Z^{\infty}$ are the classes $\lambda^{c\mG_k}_\bull$ for $k\geq1$.
\end{thm}
\begin{proof}
    For the factors $\Z/2$ and $\Z$ see \cite{K-Dax2}.
    In \cite{BG} and \cite{K-Dax} it was shown that certain ``self-referential'' graspers, depicted in Figure~\ref{F:self-ref-grasper}, generate the $\Z^{\infty}$ summand. In Figure~\ref{F:leaf-slide} we show an isotopy from our graspers $c\mG_k$ to those self-referential graspers, finishing the proof. Namely, the leaf $c\mG_k(L)$ can be moved from being meridian to the tube (a subset of $S$), to being meridian to the bar.
\end{proof}
\begin{figure}[!htbp]
    \labellist                       
    \tiny\hair 2pt
    \pinlabel \blue{$k$} at 43 50 
    \endlabellist
    \centering
\includegraphics[width=.23\linewidth]{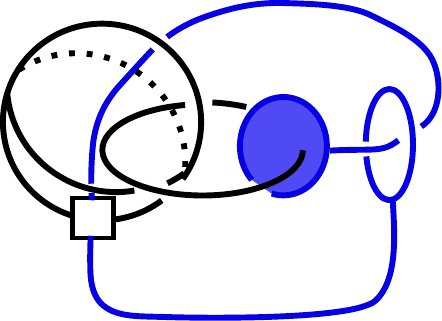}
    \caption{The self-referential grasper on $c\colon S^1\into S^1\tm D^3$ for $k\geq1$.}
    \label{F:self-ref-grasper}
\end{figure}
\begin{figure}[!htbp]
    \labellist                             % <- !
    \tiny\hair 2pt
    \pinlabel \blue{$k$} at 39 63
    \pinlabel \blue{$k$} at 393 63 
    \pinlabel \blue{$k$} at 748 63 
    \endlabellist
    \centering
\includegraphics[width=\linewidth]{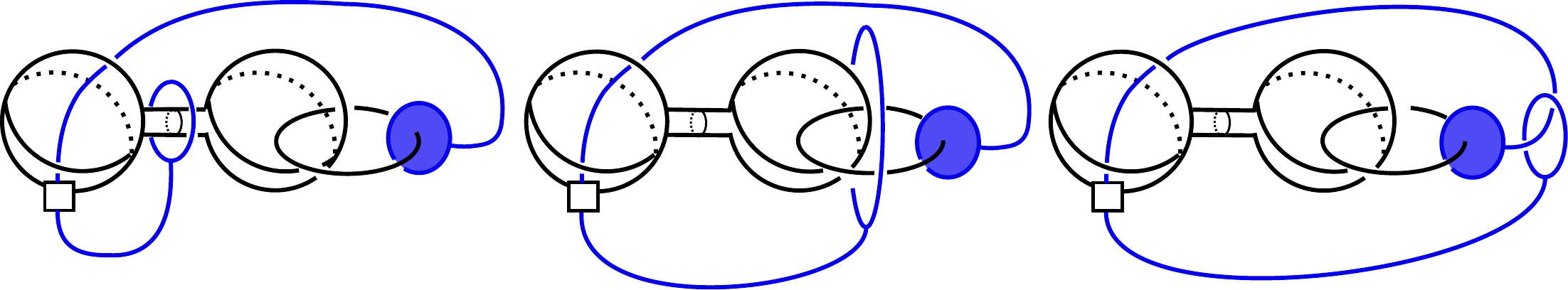}
    \caption{An isotopy of the sphere $L$ in $S^4\sm\nu (S\sqcup c)$.}
    \label{F:leaf-slide}
\end{figure}

%%%%%%%%%%%
\subsection{Completing the proof of Theorem B}\label{subsec:identify}

Recall that in Section~\ref{sec:diagram} we have proven Theorem~\ref{mainT:diagram} using the fibrations $\res_{\nu c}$ and $\res_{\nu S}$. The main diagram was in Theorem~\ref{T:main-diagram}, with all maps explicitly named; in particular, we have the connecting maps $\delta_{\nu c}^{S^1\tm D^3}$ and $\delta_{\nu S}^{S^2\tm D^2}$ of these fibrations. The following lemma finishes the proof of Theorem~\ref{mainT:infinite-subgroups}.
\begin{lem}
    The barbell diffeomorphism $\phi^{\mB_k}\in\pi_0\Diffp(S^1\tm S^1\tm I)$ agrees with the image of the loop $\lambda^{c\mG_k}_\bull$ of framed circles in $S^1\tm D^3$ under the map $\delta_{\nu c}^{S^1\tm D^3}$, and with the image of the loop $\lambda^{S\mG_k}_\bull$ of framed spheres in $S^2\tm D^2$ under the map $\delta_{\nu S}^{S^2\tm D^2}$. That is,
    \[
    \phi^{\mB_k}=\delta_{\nu c}^{S^1\tm D^3}(\lambda^{c\mG_k}_\bull)=\delta_{\nu S}^{S^2\tm D^2}(\lambda^{S\mG_k}_\bull).
    \]
\end{lem}
\begin{proof}
    By definition each of these connecting maps lifts a loop of embeddings into an ambient isotopy of the manifold in question, and then restricts to the complement of the basepoint embedding.
    
    Now simply observe that the ambient isotopy $\phi^{\mB_k}_\bull$ of $M=S^1\tm D^3$ from Lemma~\ref{L:b-nullistpy} precisely extends $\lambda^{c\mG_k}_\bull$. Indeed, on the image of $c\mG_k$ this is given as the ambient isotopy $b_\bull$ of the twirling family $\nu \alpha^1_\bull$, and elsewhere as the identity.
    The diffeomorphism $\delta_{\nu c}^{S^1\tm D^3}(\lambda^{c\mG_k}_\bull)$ is by definition the restriction of $\phi^{\mB_k}_1$ to $S^1\tm D^3\sm\nu c$, and this is exactly $\phi^{\mB_k}$. 
    
    The argument for $\delta_{\nu S}^{S^2\tm D^2}(\lambda^{S\mG_k}_\bull)=\phi^{\mB_k}$ is analogous, and uses Lemma~\ref{L:twirl}.
\end{proof}

\begin{rem}\label{rem:self-ref-barbell}
    The isotopy of the grasper $c\mG_k$ from Figure~\ref{F:leaf-slide} gives an isotopy of $\mB_k$ as well: just erase the ball $Q_2$ from all the pictures. This gives an alternative ``self-referential'' description of this barbell that will be used later (see for example Figure~\ref{F:bb after surgery}).
\end{rem}

\begin{cor}
\begin{itemize}
\item 
    The abelian group $\pi_0\Diffp(S^1\tm S^2\tm I)$ has an infinitely generated free abelian summand $\Z^\infty<\pi_0\Diffp(S^1\tm S^2\tm I)$. Its generators are the barbell diffeomorphisms $\phi^{\mB_k}$ for the barbells $\mB_k$ from Figure~\ref{F-intro}(i).

\item
    The abelian group $\pi_1(\Embp(\nu S^2,S^2\tm D^2),\nu S)$ has an infinitely generated free abelian summand $\Z^\infty<\pi_1(\Embp(\nu S^2,S^2\tm D^2),\nu S)$.
    Its generators are the $S\mG_k$-grasper families  $\lambda^{S\mG_k}_\bull$ depicted in Figure~\ref{F-intro}(ii).
\end{itemize}
\end{cor}

%%%%%%%%%%%%%%%%%%%%%%
\section{Pseudo-isotopy obstructions}\label{sec:PI}

In this section our goal is to show that the nontriviality of the barbell diffeomorphisms $\phi^{\mB_k}$ of $S^1\tm S^2\tm I$ is also detected by Hatcher and Wagoner's second obstruction $\Theta$ to the ``pseudo-isotopy implies isotopy'' problem \cite{HW}.  Throughout this section we denote
\[
    X\coloneqq S^1\tm S^2\tm I
\]
The existence of an infinitely generated subgroup of $\pi_0 \Diffp(X)$ that is detected by $\Theta$ was first shown by Singh \cite{Singh}. He starts by building for a general $4$-manifold appropriate 1-parameter families of 2-spheres whose $\Theta$ realizes certain nontrivial values, thus generating pseudo-isotopies that are not isotopic to isotopies. He then shows that for $X$ these values survive in the suitable quotient (by the ``inertia subgroup''), and thus obstruct the induced diffeomorphism of $X$ from being isotopic to the identity. 

In contrast to Singh, we begin with an explicit diffeomorphism $\phi^{\mB_k}$ for a fixed $k\geq1$ and in Section~\ref{subsec:BB-to-S2} we construct a loop $R_\bull$ of $2$-spheres in $X \# S^2\tm S^2$, that gives rise to a pseudo-isotopy for $\phi^{\mB_k}$. In fact, in Section~\ref{subsec:BB-to-double-BB} we first describe a diffeomorphism $\wt{\phi}$ of $X \# S^2\tm S^2$ that is, on one hand, isotopic to $\phi^{\mB_k}$ after surgery on $S^2\tm \{p\}$, and on the other hand easily gives rise to a suitable family $R_\bull$. In Section~\ref{subsec:HWS} we recall the definition of $\Theta$, and in Section~\ref{subsec:compute} we compute it for the family $R_\bull$.

%%%%%%%%%%%
\subsection{From barbells to barbell pairs}\label{subsec:BB-to-double-BB}
Our first step is to define a pair of barbells $\wt{\mB}_k^{\pm}$ in $X \# S^2\tm D^2$, shown in Figure~\ref{F:double bb} and explained as follows. 

Recall that in Figure~\ref{F-intro}(i) the barbell $\mB_k$ is drawn in blue in $S^4\sm\nu(S\sqcup c)$, which we identify with $X$. In Figure~\ref{F:double bb} apart from this black Hopf link $S\sqcup c$ we have another Hopf link, consisting of the red sphere $R$ and the light green circle $\gamma$. We can identify
\[
    X \# S^2\tm D^2 \;\cong\; X\sm \nu \gamma \;=\; S^4\sm\nu(S\sqcup c\sqcup \gamma),
\]
so that $R$ is sent to $S^2\tm\{0\}$ and the obvious green disk $D_G$ bounded by $\gamma$ is identified with $\{p\}\tm D^2$. We will also use the inclusion
\[
    X \# S^2\tm D^2 \;\subset\; X \# S^2\tm S^2
\]
that adds to $X\sm \nu \gamma$ a copy of $S^2\tm D^2$ not drawn in the picture.
\begin{figure}[!htbp]
    \labellist                             % <- !
    \hair 2pt
    \pinlabel $k$ at 57 80  
    \endlabellist
    \centering
    \includegraphics[width=0.4\linewidth]{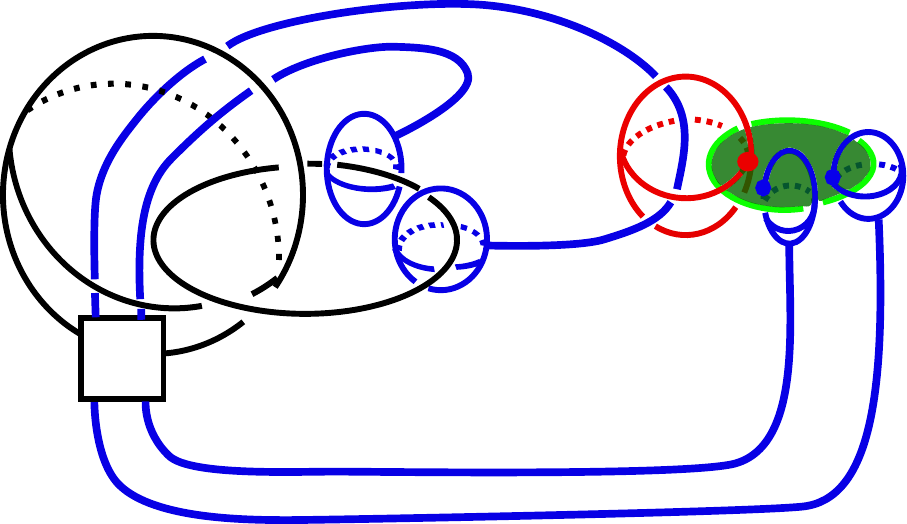}
    \caption{The manifold $S^1\tm S^2\tm I \# S^2\tm D^2$ and the barbells $\wt{\mB}_k^{\pm}$ in blue.}
    \label{F:double bb}
\end{figure}
\begin{rem}\label{R:cobordism}
    The choice of $D_G$ made in identifying $X \sm\nu \gamma$ with $X \# S^2\tm D^2$ gives rise to a cobordism from $X$ to $X \# S^2\tm S^2$ with a single 5-dimensional 2-handle. The cobordism is constructed by starting with $X\tm I$ and $D_G\subset X\tm \{1\}$, pushing the interior of $D_G$ into the interior of $X\tm I$ and removing a neighborhood. What is left of $X\tm \{1\}$ can be identified with $X \sm \nu \gamma$, so that the original $D_G$ that still resides in $X \sm \nu \gamma$ corresponds to a hemisphere of the belt sphere $G$ of the $2$-handle (whereas the neighborhood of the other hemisphere is the mentioned missing $S^2\tm D^2$). For details on such constructions we refer the reader to \cites{milnorLHCT, Gay} for example.
\end{rem}

Finally, in Figure~\ref{F:double bb} we have two barbells $\wt{\mB}_k^{+}$ and $\wt{\mB}_k^{-}$, that differ only in that $\wt{\mB}_k^{+}$ links with $R$ whereas $\wt{\mB}_k^{-}$ does not; as before, $k\geq1$ in the box indicates the linking number of the bar and the sphere $S$, and is fixed throughout the section.
\begin{defn}
    Let $\phi^{\wt{\mB}_k^{\pm}}\in\Diffp(X \# S^2\tm D^2)$ be the barbell diffeomorphism (as in Definition~\ref{D: barbell-diffeo}) on the barbell $\wt{\mB}_k^{\pm}\subset X \# S^2\tm D^2$. Moreover, let us define
\begin{equation}\label{eq:bar-phi}
   \wt{\phi} \coloneqq  \phi^{\wt{\mB}_k^{+}}\circ (\phi^{\wt{\mB}_k^{-}})^{-1} \,\in\Diffp(X \# S^2\tm D^2).\qedhere
\end{equation}
\end{defn}

The following is the first key property of $\wt{\phi}$. Note that starting from $X \# S^2\tm S^2$, we can return back to $X$ by doing surgery along the way red sphere $R$, that is, we remove a neighborhood $\nu R$ and glue in $D^3\tm S^1$ along the boundary. 

\begin{prop}\label{P:diffeo-after-surgery}
    After surgery along $R$, the diffeomorphism $\wt{\phi}$ becomes a diffeomorphism of $X$ isotopic to $\phi^{\mB_k}$.
\end{prop}
\begin{proof}
    Firstly, note that the barbell diffeomorphisms $\phi^{\wt{\mB}_k^{\pm}}$ descend to the surgered manifold, since they have supports disjoint from $R$.
    The effect that surgery has on Figure~\ref{F:double bb} is to erase the red-green Hopf link and modify the barbells $\wt{\mB}_k^{+}$ and $\wt{\mB}_k^{-}$ to the pair of barbells $\mB_k^{+}$ and $\mB_k^{-}$ shown in Figure~\ref{F:bb after surgery}. The depicted change comes from sliding the bar of $\wt{\mB}_k^{+}$ over the other belt sphere hemisphere $G\sm D_G$. This transfers the linking of the bar and $R$ in Figure~\ref{F:double bb} to the linking of the bar and the barbell cuffs in Figure~\ref{F:bb after surgery}. Hence, in the surgered manifold $\wt{\phi}$ is isotopic to the composition $\phi^{\mB_k^{+}}\circ (\phi^{\mB_k^{-}})^{-1}$.
\begin{figure}[!htbp]
    \labellist                             % <- !
    \hair 2pt
    \pinlabel $k$ at 57 85  
    \endlabellist
    \centering
    \includegraphics[width=0.37\linewidth]{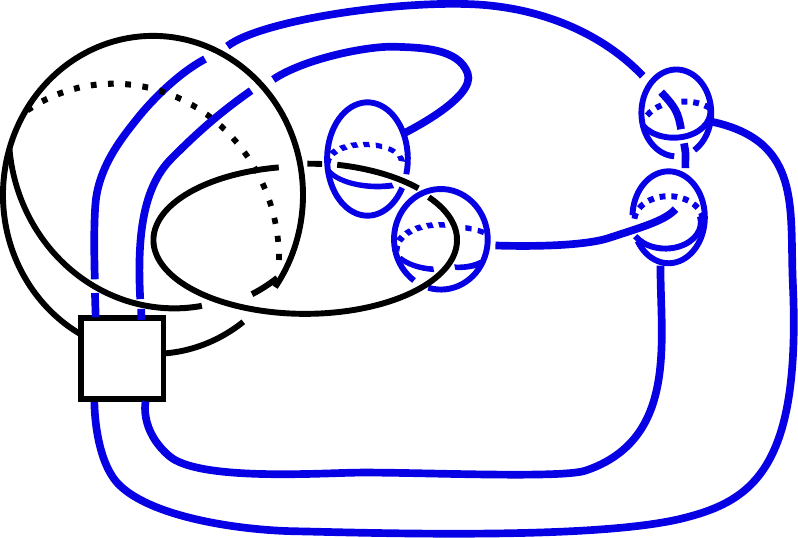}
    \caption{The pair of barbells $\wt{\mB}_k^{\pm}$ after surgery on $R$.}
    \label{F:bb after surgery}
\end{figure}

    Now thinking of the surgered manifold as $X$, notice that $\mB_k^{+}$ is precisely the self-referential version from Remark~\ref{rem:self-ref-barbell} of our barbell $\mB_k$. 
    On the other hand, the barbell $\mB_k^{-}$ extends to a grasper (one of its cuffs bounds a 3-ball with interior disjoint from $\mB_k^{-}$), so by Lemma~\ref{L:b-nullistpy} the diffeomorphism $\phi^{\mB_k^{-}}$ of $X$ is isotopic to the identity. Therefore, after surgery $\wt{\phi}$ is isotopic to $\phi^{\mB_k}$.
\end{proof}

%%%%%%%%%%%
\subsection{From barbell pairs to loops of 2-spheres}\label{subsec:BB-to-S2}
The second key property of $\wt{\phi}$ is that is in fact isotopic to the identity! The barbells $\wt{\mB}_k^{+}$ and $\wt{\mB}_k^{-}$ differ only in the linking of the bar with the sphere $R$, but this is irrelevant in $X \# S^2\tm D^2\cong X\sm \nu \gamma$. 

More precisely, a path of diffeomorphisms $\wt{\phi}_\bull$  from the identity to the diffeomorphism $\wt{\phi}$ is given by isotoping the support of $\phi^{\wt{\mB}_k^{+}}$ to agree with that of $\phi^{\wt{\mB}_k^{-}}$. That is, we use an isotopy $\wt{\mB}_k^t$ from the barbell $\wt{\mB}_k^{+}$ to $\wt{\mB}_k^{-}$. More concretely, we first unlink $\wt{\mB}_k^{+}$ from $R$ using an \emph{unlinking homotopy} $\wt{\mB}_k^t$,  $t\in [0,1]$, as in Figure \ref{F:bb isotopy}.
Then there is an isotopy $\wt{\mB}_k^t$,  $t\in [1,2]$, taking $\wt{\mB}_k^1$ to $\wt{\mB}_k^2 = \wt{\mB}_k^{-}$, and supported in $X \# S^2\tm D^2\sm \nu R$. Finally, we define
\begin{equation}
    \wt{\phi}_t = \phi^{\wt{\mB}_k^{2-2t}}\circ (\phi^{\wt{\mB}_k^{-}})^{-1} \,\in\Diffp(X \# S^2\tm D^2).
\end{equation}
for $t\in[0,1]$. We have reversed and scaled the time parameter so that $\wt{\phi}_0 = \id$, $\wt{\phi}_1 = \wt{\phi}$.
\begin{figure}[!htbp]
    \centering
    \includegraphics[width=0.5\linewidth]{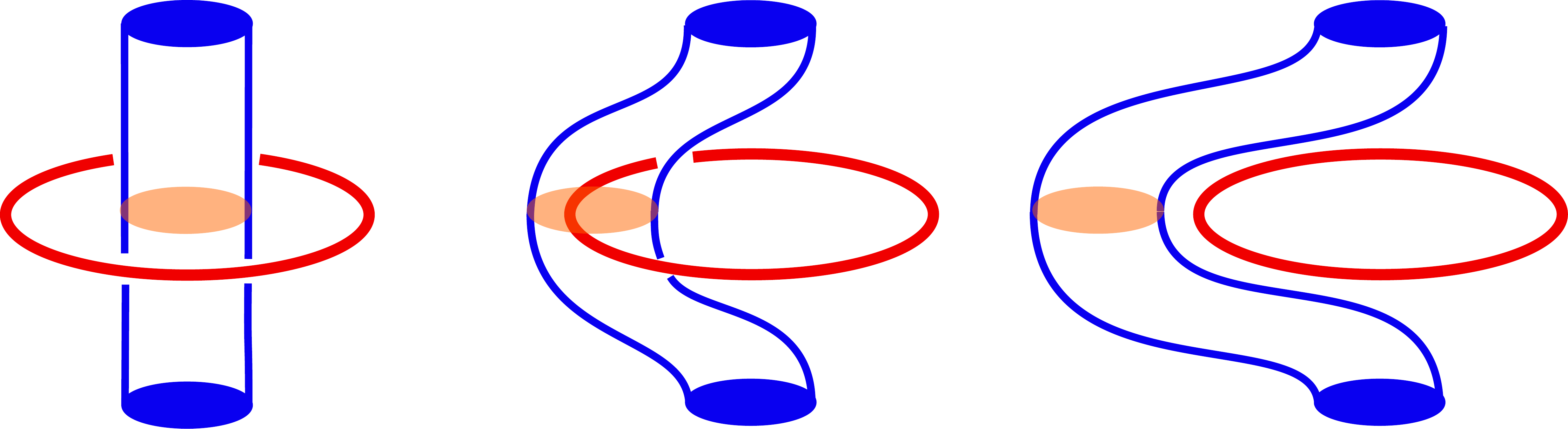}
    \caption{Unlinking homotopy.}
    \label{F:bb isotopy}
\end{figure}
\begin{defn}
    For $t\in[0,1]$ let us define
    \[
    R_t \coloneqq \wt{\phi}_t\circ R\colon\; S^2\into X \# S^2\tm D^2.
    \]
    This gives a loop of embeddings with $R_0=R_1=R$, since $\wt{\phi}_1(R)=\wt{\phi}(R)=R$ as the barbells $\wt{\mB}_k^\pm$ are both disjoint from $R$. 
\end{defn}

\begin{rem}\label{R:Stable isotopy}
    Note that $\wt{\phi}_\bull$ is by definition an ambient isotopy lifting $R_\bull$. 
    Moreover, note that Proposition~\ref{P:diffeo-after-surgery} implies that the path $\wt{\phi}_\bull$ considered in $X \# S^2\tm D^2\subset X \# S^2\tm S^2$ is a stable isotopy from the identity to $\phi^{\mB_k}$.
\end{rem}

By the definition of $\wt{\phi}_\bull$, we have that $R_t\neq R$ only when $\frac{1}{2}\leq t\leq 1$, that is, during the unlinking homotopy. In order to describe the spheres $R_t$ we now make a specific choice of this homotopy. We will assume that the intersections of the barbell and $R$ occur precisely in the midball $\Delta\colon D^2\tm D^1\into \sB$; recall from Definition~\ref{D:models} that the cross-sections $\Delta_s$ are equal to the intersection of $\sB$ with the $yz$-plane in $\R^3\tm \{s\}$. The $yz$-coordinates then gives a natural foliation of the interior of $D^2\tm \{s\}$ as $D^1\tm D^1\tm \{s\}$ for each $s$.

We can ensure that the thickened bar is transverse to $R$ except at two times $t_{I},t_{II}\in[0,1]$, when they touch for the first and last time. Moreover, there exists an $\epsilon>0$ so that for all $t\in [t_{I}+\epsilon,t_{II}-\epsilon]$, we have
\[
    \hat{R}_t\coloneqq(\wt{\mB}_k^t)^{-1}(R) = \big(\{y(t)\}\tm D^1\big)\tm D^1 \subset \Delta,
\]
where $y(t)$ is a strictly decreasing function of $t$. This is depicted in Figure~\ref{F:R-hat}: as $t$ varies in the interval $[t_{I}+\epsilon,t_{II}-\epsilon]$ the red arcs smoothly sweep across the orange disk in each cross-section simultaneously. The union of all $\hat{R}_t$ for  $t\in[t_{I}+\epsilon,t_{II}-\epsilon]$ is equal to $\Delta$.
\begin{figure}[!htbp]
    \centering
    \includegraphics[width=\linewidth]{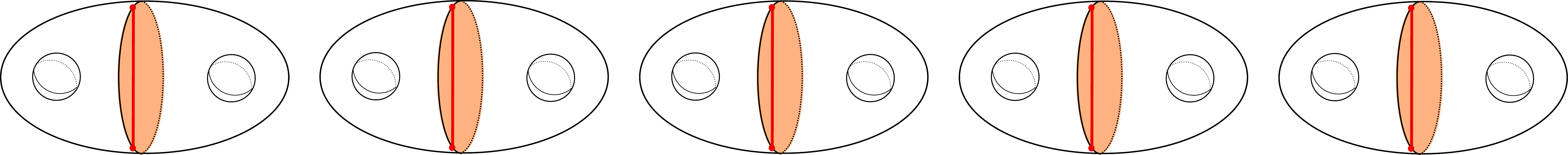}
    \caption{Cross-sections $\Delta_s\subset \sB_s$ for several values of $s\in D^1$. Depicted in red is $\hat{R}_t=(\wt{\mB}_k^t)^{-1}(R)$ for a fixed $t\in [t_{I}+\epsilon, t_{II}-\epsilon]$.}
    \label{F:R-hat}
\end{figure}

The map $\wt{\phi}_t$ leaves $R\sm \wt{\mB}_k^{2-2t}$ intact, and moves $R\cap \wt{\mB}_k^{2-2t}=\wt{\mB}_k^{2-2t}(\hat{R}_{2-2t})$ using the barbell map. Therefore, for $t\in[t_{I}+\epsilon,t_{II}-\epsilon]$ we have 
\begin{equation}\label{E: R_t}
R_t= 
    \wt{\mB}_k^{2-2t}(b(\hat{R}_{2-2t}))\cup (R\sm \wt{\mB}_k^{2-2t}).
\end{equation}
The set $b(\hat{R}_{2-2t})$ is depicted in Figure~\ref{F:loop R_t}. For $t\notin [t_{I}+\epsilon,t_{II}-\epsilon]$, we have $R_t=R_0=R$.
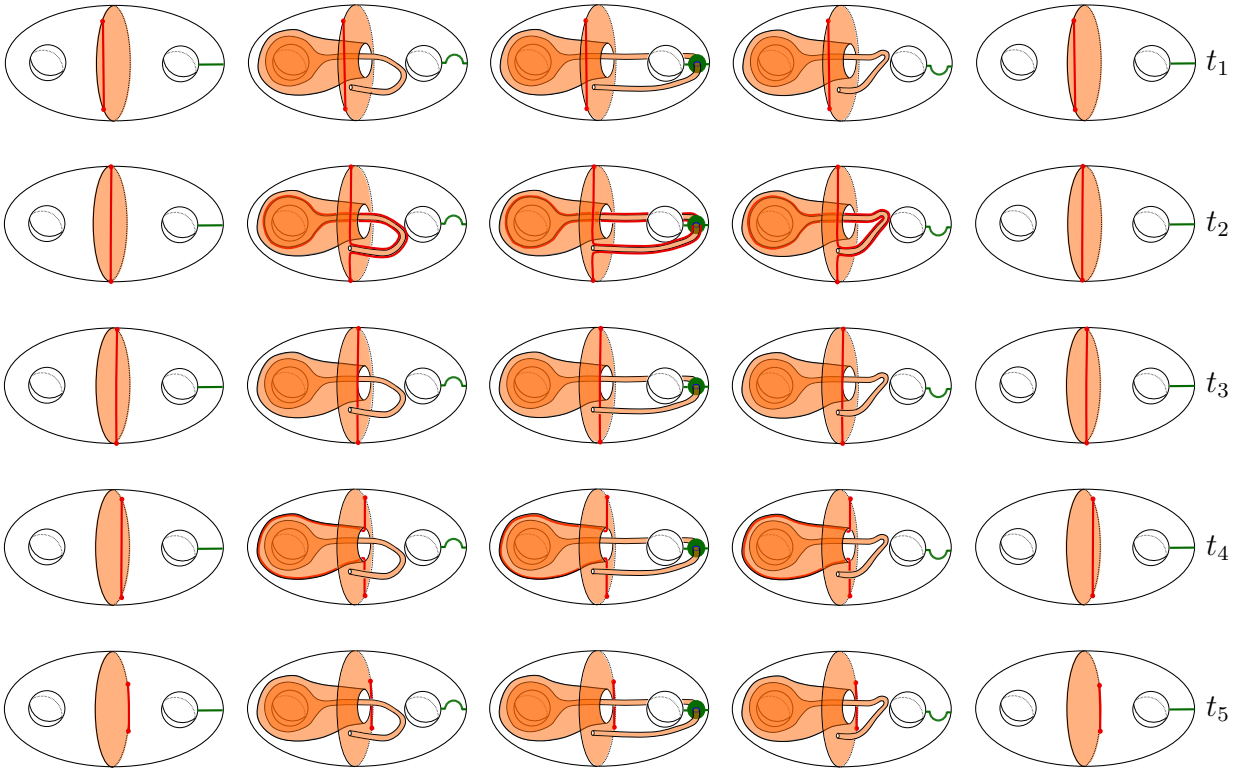
\begin{figure}[!htbp]
    \labellist                             % <- !
    \hair 2pt
    \pinlabel { $t_1$} at 3200 1860
    \pinlabel { $t_2$} at 3200 1435
    \pinlabel { $t_3$} at 3200 1010
    \pinlabel { $t_4$} at 3200 580
    \pinlabel { $t_5$} at 3200 155
    \endlabellist
    \centering
    \includegraphics[width=\linewidth]{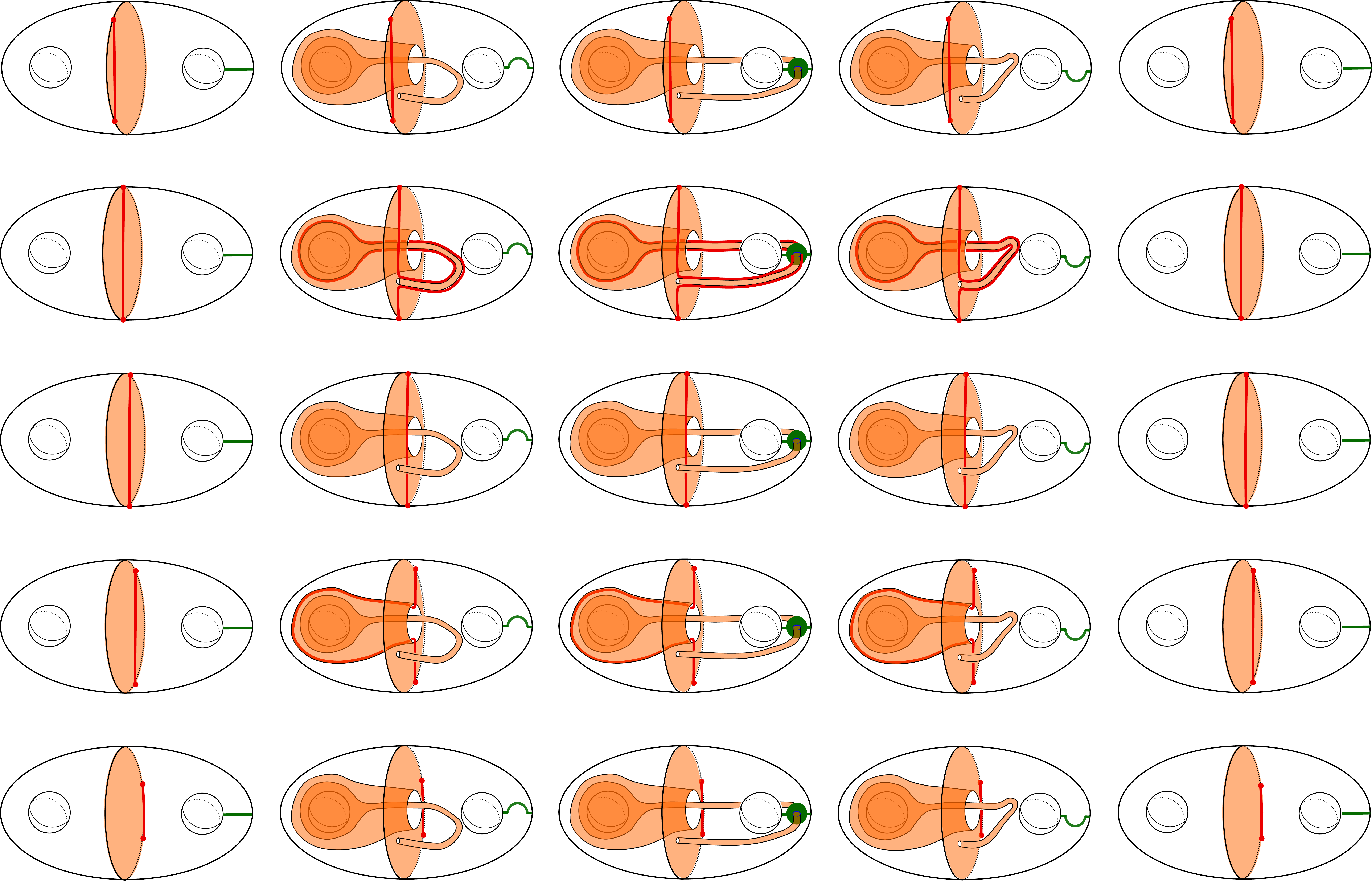}
    \caption{Each row corresponds to a fixed moment $t_i\in[t_{I}+\epsilon,t_{II}-\epsilon]$ and depicts cross-sections $\sB_s$ for several $s\in D^1$ together with the cross-sections of $b(\Delta)$ in orange, and $b(\hat{R}_{2-2t})$ in red.}
    \label{F:loop R_t}
\end{figure}

%%%%%%%%%%%
\subsection{The work of Hatcher, Wagoner, and Singh}\label{subsec:HWS}

In this section we give a terse review of pseudo-isotopy theory, Hatcher and Wagoner's second obstruction $\Theta$, as well as Singh's work needed for our computation. We refer the reader to \cites{cerf1970, HW, Hatcher-PI, Singh, Gay, gabai2024} for details.

Let $\mP_\partial(M)$ denote the space of pseudo-isotopies of a 4-manifold $M$, that is, diffeomorphisms of $M\tm I$ that restrict to the identity on $M\tm \{0\} \cup \partial M \tm I$. Note that the trace of an ambient isotopy from the identity is a level preserving pseudo-isotopy, and vise versa. To study the problem of when a pseudo-isotopy is isotopic to an isotopy for simply connected closed manifolds of dimension $n\geq 5$, Cerf~\cite{cerf1970} related isotopy classes $[F]\in\pi_0 \mP_\partial(M)$ of pseudo-isotopies to relative homotopy classes of paths of generalized Morse functions and gradient-like vector fields. In this setting he was able to show that every pseudo-isotopy is isotopic to an isotopy. 

Later, Cerf's approach was generalized to the nonsimply connected case by Hatcher and Wagoner~\cite{HW}, where they gave two obstructions to the problem. The first Hatcher--Wagoner obstruction $\Sigma$ obstructs the existence of a ``nested-eye representative'' of a pseudo-isotopy, whereas the second obstruction~$\Theta$ is defined on the classes with such a representative, and obstructs the ``cancellation of eyes''. Igusa~\cite{igusa} observed that their result needs an additional assumption -- that the first Postnikov invariant $k_1(M)$ vanishes -- and defined a correction term otherwise. 
\begin{thm}[\cite{HW,igusa}]\label{T:HW 2nd obs}
    Suppose that $k_1(M)=0$. Let $\mK < \pi_0 \mP_\partial(M)$ consist of those isotopy classes with a nested-eye representative. Then $\mK$ is in fact a subgroup and there exists a homomorphism 
    \[
        \Theta \colon\mK\longrightarrow \faktor
        {(\Z_2\tm \pi_2M)[\pi_1M]}
        {\langle \alpha \sigma - \alpha^{\tau}\tau \sigma \tau^{-1}, \alpha 1: \alpha \in \Z_2\tm \pi_2M,\sigma, \tau \in \pi_1M\rangle}\;,
    \]
    such that if $\Theta([F])\neq 0$ then $[F]\in\mK$ is not isotopic to an isotopy.
\end{thm}

Here $\alpha^\tau$ denotes the standard action of $\pi_1M$ on $\pi_2M$, and the trivial action on $\Z/2$. Important for us is the fact (see for example \cite{Gay,gabai2024}) that there is a correspondence between nested-eye representatives $F$ and families
\begin{align}\label{E: Loop of embeddings}
    R_t\colon & \sqcup_{i=1}^n S^2\into M\#^n S^2\tm S^2
\end{align}
with $t\in[0,1]$, such that $R_0 = R_1 = \sqcup_{i=1}^n S^2\tm\{p_i\}$, for $p_i$ in the $i$-th component of the connected sum. In one direction, the loop $R_\bull$ arises from the ``middle-middle level'' of $F$. In the other direction, we first attach $n$ trivial 5-dimensional 2-handles to $M\times I$, so that at the top we see the belt spheres
\begin{align}\label{E: disks}
    G\coloneqq \sqcup_{i=1}^n(\{p\}\tm S^2)\into M\#^n S^2\tm S^2,
\end{align}
and then $R_\bull$ serves as a family of attaching maps for 5-dimensional 3-handles, starting and ending in cancelling position: $R_0^i \cap G^j =R_1^i \cap G^j= \delta_{ij}$. This can be interpreted as a pseudo-isotopy from the identity to some diffeomorphism $\phi$. 

Furthermore, ambiently extending the loop $R_\bull$ produces a path of diffeomorphisms $\wt{\phi}_\bull$ such that after doing surgery on $R_1$ the diffeomorphism $\wt{\phi}_1$ agrees with $\phi$; compare to Proposition~\ref{P:diffeo-after-surgery} and see  \cite{Gay,K-graspers-mcg,gabai2024}.

To compute $\Theta$ one uses the intersection locus of $R_\bull$ and $G$, defined as follows.
\begin{defn}\label{D:locus}
    Let $F\in\mP_\partial(M)$ be a a nested-eye pseudo-isotopy. Let $R_\bull$ be the associated loop of embeddings in $M\#^n S^2\tm S^2$ as in~\eqref{E: Loop of embeddings}, with the corresponding dual embedding $G$ as in~\eqref{E: disks}. Define the intersection locus
\[
    \mL_{G^j}(R_\bull^i) \coloneqq \big\{\, (p,t) \in (M\# S^2\tm S^2)\tm I~|~p\in R_t^i \cap G^j\subset (M\# S^2\tm S^2)\tm \{t\}\,\big\}.\qedhere
\]
\end{defn}

\begin{rem}
    Assuming general position, and fixing whiskers for $R^i$ and $G^j$, the space $\mL_{G^j}(R_\bull^i)$ becomes a based framed 1-manifold.
    The following is a sketch definition of $\Theta([F])$; we refer to \cite{HW, Singh} for details. For every $1\leq i,j\leq n$ the class $a_{ij}\in (\Z_2\tm \pi_2M)[\pi_1M]$ is the sum over all closed components $C\in \mL_{G^j}(R_\bull^i)$ of certain associated twisted spheres $\alpha_C \in \Z_2\tm \pi_2M$ and  group elements $\sigma_C \in \pi_1M$.
    The determines an element $\Theta(F)=[I+(a_{ij})]\in \operatorname{Wh}_1(\pi_1M;\Z_2\tm \pi_2M)$.
    This Whitehead group was shown in \cite{Hatcher-PI} to be isomorphic to the target of $\Theta$ from Theorem~\ref{T:HW 2nd obs}. %so that $\Theta(F)=\sum_{i=1}^na_{ii}$.
\end{rem}

Calculating $\Theta([F])$ for general $[F]\in \mK$ and $M$ is difficult.
However, when $[F]$ is determined by a sufficiently nice loop $R_\bull$ there is a simple definition of $\Theta$, as follows.

\begin{defn}[\cite{Singh}]
    A loop $R_\bull$ associated to $F\in\mP_\partial(M)$ as above is called a \emph{geometrically simple loop} if the embedding $R_t\colon S^2\into M\# S^2\tm S^2$ is transverse to $G$ except at two distinct instances $\{t_f, t_w\}$ with $0<t_f<t_w<1$. These moments happen during a finger (resp.\ Whitney) move: just after $t_f$ (resp.\ $t_w$), the total number of intersections between $R_t$ and $G$ has increased (res.\ decreased) by two. 
\end{defn}
\begin{figure}[!htbp]
    % \labellist                       
    % \small\hair 2pt
    % \pinlabel \emph{finger move arc} [l] at -100 100
    % \pinlabel \blue{\emph{Whitney Disk}} at 2500 600
    % \endlabellist
    \centering
    \includegraphics[width=0.5\linewidth]{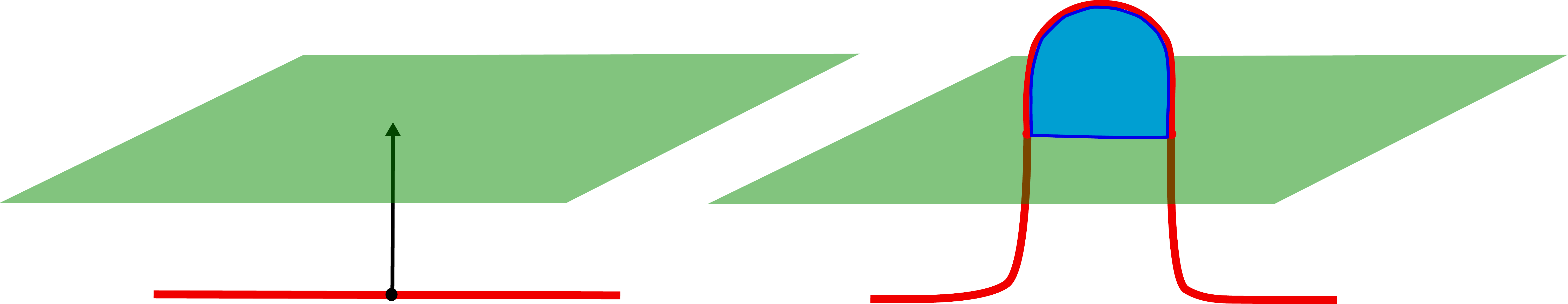}
    \caption{The finger and Whitney moves.}
    \label{F:FW move}
\end{figure}
The local picture is shown in Figure~\ref{F:FW move}: by moving from left to right one sees the finger move along the black \emph{finger move arc}, creating two new intersections, and from right to left a Whitney move along the framed blue \emph{Whitney disk}, removing the two intersections. In this local picture, these are inverse operations: just after a finger move is performed, there is a framed Whitney disk, called the \emph{finger disk}, which reverses the finger move.

\begin{thm}[\cite{Singh}]\label{T:Singh}
    If $R_\bull$ is a simple loop of embeddings, then there is at most one circle of intersections $C$ in $\mL_{G}(R_\bull)$. Let $\sigma_C\in\pi_1M$ be the loop given by following the finger move arc and the whiskers for $R_\bull$ and $G$.
    Let $FW\in\pi_2M$ be the union of the finger and Whitney disks, and let $\epsilon\in\Z_2$ be the difference mod two of the Whitney sections of the two disks along their common boundary. Let $\alpha_C=(\epsilon,FW)\in\Z_2\tm \pi_2M$. Then
    \[
        \Theta([F])=[\alpha_C \sigma_C] \in  \faktor{(\Z_2\tm \pi_2M)[\pi_1M]}{\langle \alpha \sigma - \alpha^{\tau}\tau \sigma \tau^{-1}, \alpha1\rangle}\,.
    \]
\end{thm}

%%%%%%%%%%%
\subsection{Completing the proof of Theorem C}\label{subsec:compute}

Singh computed the target of $\Theta$ explicitly when $M=X\coloneqq S^1\tm S^2\tm I$, and showed which classes obstruct the induced diffeomorphism of $S^1\tm S^2\tm I$ from being trivial in $\pi_0\Diffp(S^1\tm S^2\tm I)$. 
Let us identify $\pi_2X\cong\Z$, and $\pi_1X\cong\Z$ generated by a class $x$.

\begin{thm}[\cite{Singh}]\label{T:Singhs diffeomorphism result}
    For the manifold $X=S^1\tm S^2\tm I$, we have
\begin{equation*}
    \faktor{(\Z_2\tm \pi_2X)[\pi_1X]}{\langle \alpha \sigma - \alpha^{\tau}\tau \sigma \tau^{-1}, \alpha1\rangle} \cong \faktor{(\Z_2\tm \Z)[x, x^{-1}]}{\langle 1\rangle}\;.
\end{equation*}
    Further, if for a diffeomorphism $\phi$ there is a pseudo-isotopy $F$ to the identity, such that $q\circ\Theta([F])$ is nontrivial for the quotient homomorphism
    \[
    q\colon\faktor{(\Z_2\tm \Z)[x, x^{-1}]}{\langle 1\rangle}
    \ra \faktor{(\Z_2\tm \Z)[x, x^{-1}]}{\langle 1, x^k-x^{-k}\rangle}\;,
    \]
    then $\phi$ is not isotopic to the identity. That is $\phi\neq0\in\pi_0\Diffp(S^1\tm S^2\tm I)$.
\end{thm}

In order to compute the Hatcher--Wagoner invariant $\Theta$ for our barbell diffeomorphisms $\phi^{\mB_k}$, we need to analyze the intersection locus $\mL_{G}(R_\bull)$ (see Definition~\ref{D:locus}) between the belt sphere $G$ our loop $R_\bull$ defined in Section~\ref{subsec:BB-to-S2}. We first determine the intersection set. 

\begin{lem}\label{L: locus}
   We have $\bigcup_{t\in[0,1]}(G\cap R_t) = (D_G\cap R) \cup (D_G\cap\mB^{+}_k(b(\Delta)))$.
\end{lem}
\begin{proof}
    The hemisphere $G\sm D_G$ is not used in the construction of $R_\bull$ so $G\cap R_t=D_G\cap R_t$.
    By~\eqref{E: R_t} we have $R_t=\wt{\mB}_k^{2-2t}(b(\hat{R}_{2-2t}))\cup (R\sm \wt{\mB}_k^{2-2t})$ when $\wt{\mB}_k^{2-2t}\cap R\neq \varnothing$, and $R_t = R$ otherwise.
    Thus, each $D_G\cap R_t$ consists of $D_G\cap R$ and possibly $D_G\cap \wt{\mB}_k^{2-2t}(b(\hat{R}_{2-2t}))$. 
    
    From Figure~\ref{F:double bb} we see that the only intersection of $\wt{\mB}_k^{2-2t}$ and the green disk $D_G$ is a neighborhood of the point where one cuff of $\mB^{+}_k$ intersects $D_G$, and is independent of $t$. In other words, we can replace $\wt{\mB}_k^{2-2t}$ by $\wt{\mB}_k^0=\wt{\mB}_k^{+}$. But the union of $\hat{R}_{2-2t}\subset\Delta$ for all $t$ covers the whole $\Delta$, so in total we get $D_G\cap\mB^{+}_k(b(\Delta))$.
\end{proof}
We have that $D_G\cap R=\{p\}$, so this gives a component $\{p\}\times I$ of $\mL_{G}(R)$ that will contribute nothing to the invariant $\Theta$. 
Next, we claim that $D_G\cap\mB^{+}_k(b(\Delta))$ consists of a single circle. We apply $(\mB^{+}_k)^{-1}$ to this and work inside of the model barbell $\sB$. Here $D_G$ appears as the collection of horizontal arcs (as in the leftmost picture of Figure~\ref{F:G-fixed}) in each cross-section $\sB_s$. After an isotopy of $D_G$ (a rotation in $sz$-plane) we can arrange that the intersection between $D_G$ and the knotted midball $b(\Delta)$ (given in Figure \ref{F:midball}) is confined to $b(\Delta_0)\subset\sB_0$ as in Figure~\ref{F:G-fixed}. There we see a circle of intersection where the thin orange tube pierces through the green disk.
\begin{figure}[!htbp]
    \centering
    \includegraphics[width=\linewidth]{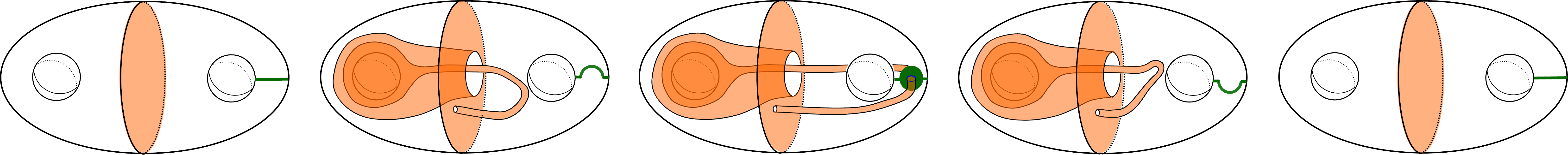}
    
    \vspace{7pt}
    
    \includegraphics[width=0.3\linewidth]{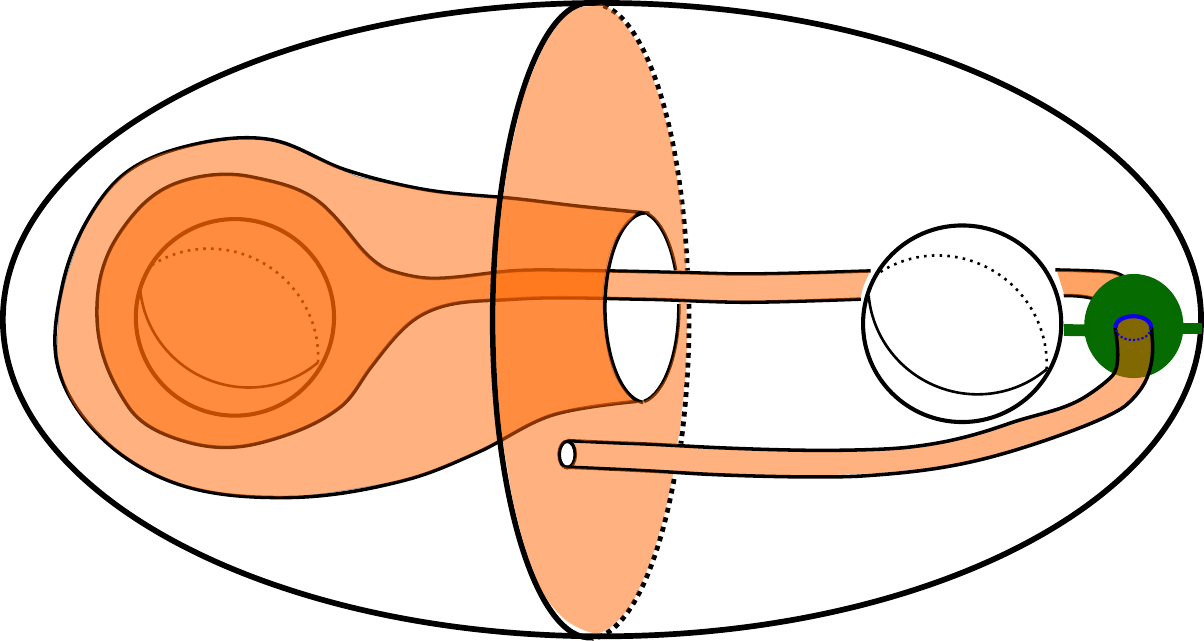}
    
    \caption{Several cross-sections $\sB_s$ of $\sB$, containing the preimage of $D_G$ (after a rotation) in green. The intersection of the preimage of $D_G$ with $b(\Delta)$ is contained in $\sB_0$, whose enlarged version is shown in the bottom row.}
    \label{F:G-fixed}
\end{figure}

Finally, observe that the foliation of $b(\Delta)$ by the arcs $R_t$ for $t\in[t_{I}+\epsilon,t_{II}-\epsilon]$, gives the foliation of the circle $D_G\cap b(\Delta)$ by pairs of points, except for two moments when it touches it for the first and last time. This implies that we obtain a single component $C$ in $\mL_{G}(R_\bull)$, apart from the mentioned arc coming from $D_G\cap R$.

\begin{thm}\label{T:BB computation}
    Let $R_t = \wt{\phi}_t(R)\colon S^2\into X \# S^2\tm D^2 \subset X \# S^2\tm S^2$ be the loop from~\eqref{E: R_t} and $C$ the component of $\mL_{G}(R)$ described above. Then $\alpha_C = (0,1)\in \Z_2\tm\pi_2X\cong\Z/2\tm \Z$ and $\sigma_C = x^k\in\langle x\rangle\cong\pi_1X$. 
\end{thm}
\begin{proof}
    We analyze the preimage of $R_\bull$ in the model barbell $\sB$ as in Figure \ref{F:loop R_t}. After arranging $D_G$ as in Figure~\ref{F:G-fixed}, we can restrict our attention to the middle column of Figure~\ref{F:loop R_t}. We identify the finger move arc, finger disk and Whitney disk in the middle-middle level as in Figure~\ref{F:FW-identification}. The finger and Whitney disks clearly fit together to give a sphere parallel to the cuff~$S_1$, and the Whitney sections agree as this sphere is framed. Since $S_1$ represents the generator of $\pi_2(X)$, it follows that $\alpha_C = (0,1)\in \Z_2\tm \Z$. 

    It remains to determine $\sigma_C\in\pi_1X$. This is defined by closing the finger move arc to a loop using whiskers from each of $R$ and $G$ to the basepoint. 
    The finger move arc is depicted at the top of Figure~\ref{F:Fmove-identification}. It is homotopic to the segment of the bar of $\wt{\mB}_k$ going from the midball $\Delta$ to the cuff $S_2$, together with a segment in $S_2$ to the intersection with $G$. In the ambient picture at the bottom of Figure~\ref{F:Fmove-identification} we see that the whiskers complete the finger move arc to the loop that follows the bar and gives $\sigma_C=x^k$. 
\end{proof}
\begin{figure}[!htbp]
    \centering
    \begin{subfigure}[c]{0.4\textwidth}
        \labellist                       
        \small\hair 2pt
        \pinlabel $F$ at 130 285
        \pinlabel $W$ at 1150 285
        \pinlabel $S_1$ at 650 350
        \pinlabel $S_1$ at 440 770
        \pinlabel {\rule{1cm}{0.5pt}} at 510 350
        \pinlabel {\rule{1cm}{0.5pt}} at 785 350
        \endlabellist
        \centering
        \includegraphics[width=\linewidth]{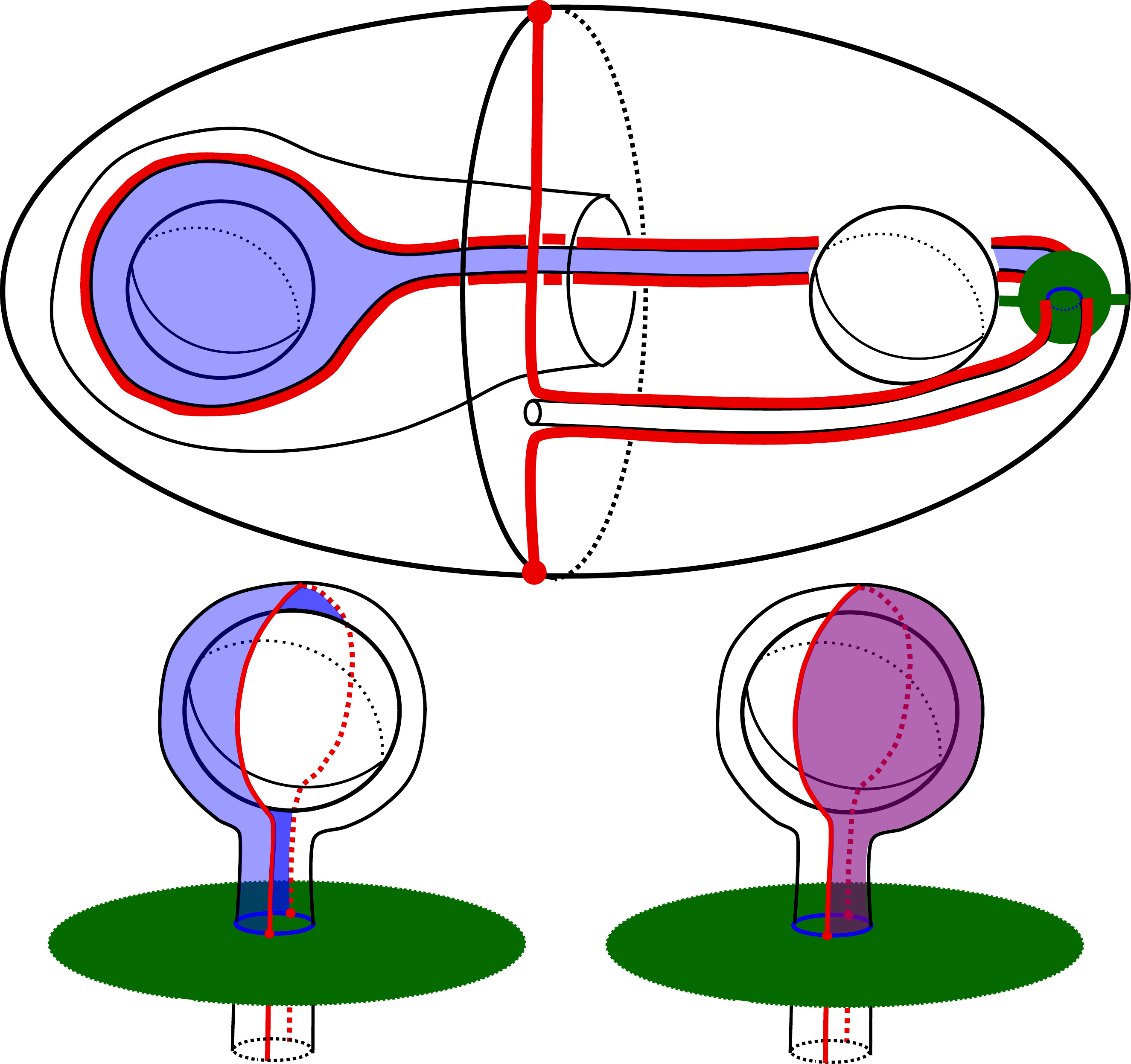}
        \caption{The identification of the finger and Whitney disks for $R_\bull$. The top picture shows the red arc $R_{t_2}\subset\sB_0$, which intersects $G$ in two points. As $R_t$ sweeps across $b(\Delta)$ backward and forward in time from $t_2$, it covers two disks -- the finger and Whitney disks. They are schematically depicted in the bottom picture, which shows that their union is $S_2$.}
        \label{F:FW-identification}     
    \end{subfigure}
    \hspace{1cm}
    \begin{subfigure}[c]{0.4\textwidth}
        \labellist                       
        \small\hair 2pt
        \pinlabel $k$ at 190 275
        \pinlabel $S_1$ at 700 350
        \pinlabel $S_1$ at 420 1180
        \pinlabel $S_2$ at 1350 550
        \pinlabel $S_2$ at 900 1180
        \endlabellist
        \includegraphics[width=0.9\linewidth]{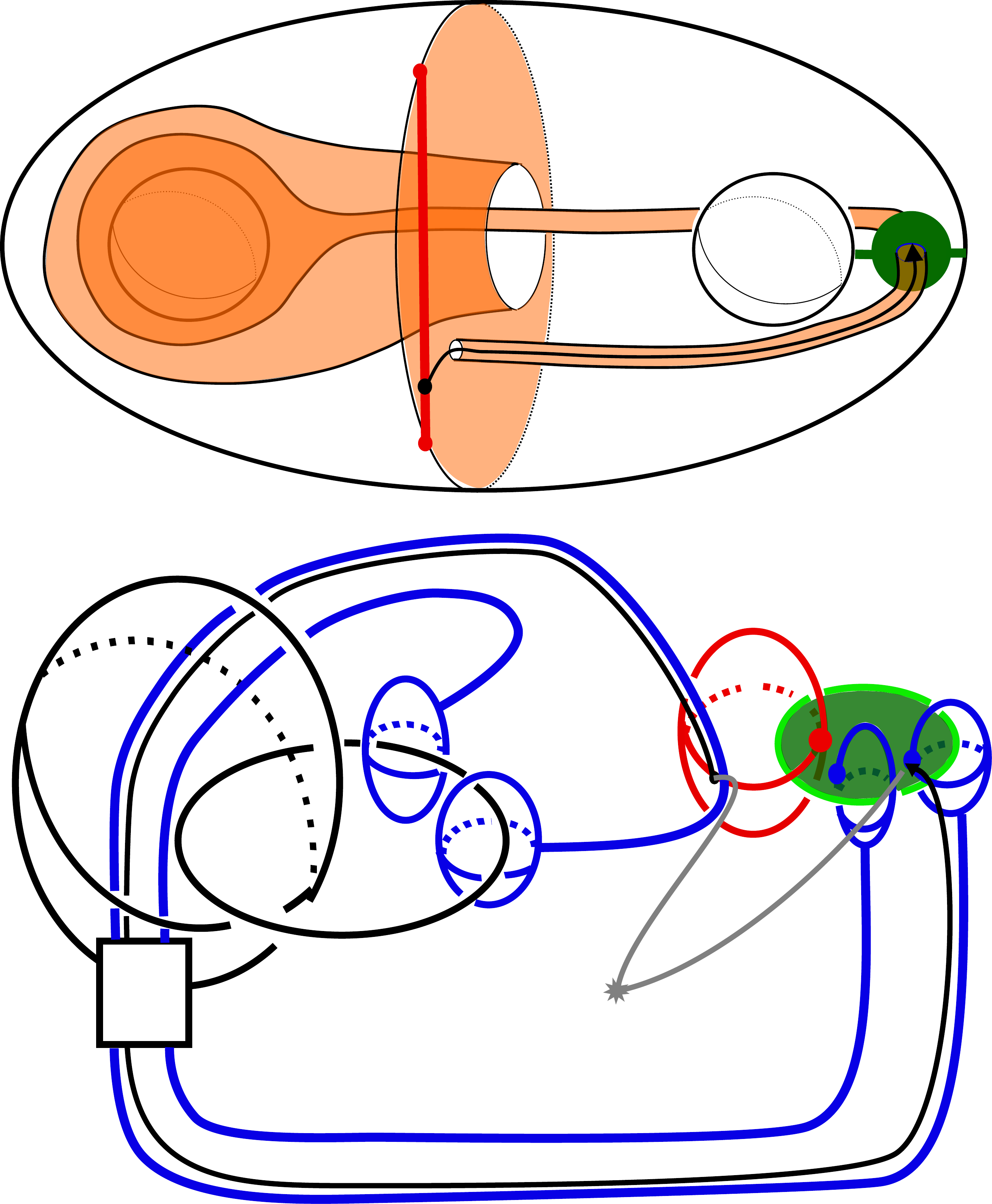}
         \caption{On top is the finger move arc in $\sB$, in black. Below is the the same arc drawn ambiently in $\mB_k^{+}$, following the bar for $\mB^k$ from $R$ to $G$, together with the gray whiskers for $R$ and $G$ to the basepoint.}
        \label{F:Fmove-identification}
    \end{subfigure}
    \caption{}
\end{figure}

We can now prove Theorem~\ref{mainT:PI}.
\begin{thm}
    For every $k\geq 1$, the barbell diffeomorphism $\phi^{\mB_k}$ is pseudo-isotopic but not isotopic to the identity, and this is detected by the Hatcher--Wagoner $\Theta$ obstruction. Moreover, the classes $\phi^{\mB_k}$ are all mutually distinct.
\end{thm}
\begin{proof}
   By Remark~\ref{R:Stable isotopy} the ambient extension $\wt{\phi}_\bull$ $R_\bull$ has the endpoint $\wt{\phi}_1=\wt{\phi}$ that is by Proposition~\ref{P:diffeo-after-surgery} isotopic to $\phi^{\mB_k}$  after surgery on $R$. By the discussion just after Theorem~\ref{T:HW 2nd obs} this means that $R_\bull$ induces a pseudoisotopy $F_k$ from the identity to the diffeomorphism $\phi^{\mB_k}$. Using Singh's Theorem~\ref{T:Singh} and the computation of Theorem \ref{T:BB computation}, we find $\Theta([F_k])= (0,1)x^k$. By Theorem~\ref{T:Singhs diffeomorphism result} these are nontrivial and distinct elements of the target group, so $F_k$ are neither isotopic to an isotopy, nor to each other. Moreover, $q$ only identifies $x^k$ with $x^{-k}$, so the classes $q\Theta([F_k])$ remain nontrivial and distinct, proving the claim. 
\end{proof}

% References

\printbibliography

% \vspace{10pt}
% \hrule

\end{document}